\documentclass{amsart}

\usepackage{amsmath, amsthm, amssymb, amstext, amsfonts}
\usepackage{enumerate}
\usepackage{graphicx}
\usepackage[T1]{fontenc} %needed for polic ogonek
\usepackage[applemac]{inputenc}

\theoremstyle{plain}

\newtheorem{thm}{Theorem}[section]
\newtheorem{lem}[thm]{Lemma}
\newtheorem{cor}[thm]{Corollary}
\newtheorem{prop}[thm]{Proposition}

\newtheorem*{prob*}{Problem}

\theoremstyle{definition}

\newtheorem{ex}[thm]{Example}

\newcommand{\R}{\ensuremath{\mathbb{R}}}
\newcommand{\N}{\ensuremath{\mathbb{N}}}
\newcommand{\Z}{\ensuremath{\mathbb{Z}}}

\newcommand{\cB}{\ensuremath{\mathcal{B}}}
\newcommand{\cD}{\ensuremath{\mathcal{D}}}

\newcommand{\cL}{\ensuremath{\mathcal{L}}}

\newcommand{\cR}{\ensuremath{\mathcal{R}}}

\newcommand{\sm}{\ensuremath{\smallsetminus}}

\newcommand{\rand}{\partial}

\newcommand{\es}{\ensuremath{\emptyset}}

\newcommand{\sub}{\subseteq}

\def\qi{quasi-iso\-metric}
\def\qiy{quasi-iso\-me\-try}
\def\qis{quasi-iso\-me\-tries}

\def\se{self-em\-bed\-ding}
\def\Se{Self-em\-bed\-ding}

\def\qg{quasi-geo\-desic}

\newcommand{\comment}[1]{}

\def\hm{semimetric}

\def\?#1{\vadjust{\vbox to 0pt{\vss\vskip-8pt\leftline{%
     \llap{\hbox{\vbox{\pretolerance=-1
     \doublehyphendemerits=0\finalhyphendemerits=0
     \hsize16truemm\tolerance=10000\small
     \lineskip=0pt\lineskiplimit=0pt
     \rightskip=0pt plus16truemm\baselineskip8pt\noindent
     \hskip0pt        %(without this, the first word is never hyphenated!)
     #1\endgraf}\hskip7truemm}}}\vss}}}

\newenvironment{txteq}
  {
    \begin{equation}
    \begin{minipage}[c]{0.85\textwidth} % set width to 0.9 x textwidth
    \em                                % switch on emph
  }
  {\end{minipage}\end{equation}\ignorespacesafterend}
\newenvironment{txteq*}
  {
    \begin{equation*}
    \begin{minipage}[c]{0.9\textwidth} % set width to 0.9 x textwidth
    \em                                % switch on emph
  }
  {\end{minipage}\end{equation*}\ignorespacesafterend}
\newenvironment{txteqtag}[1]
  {
    \begin{equation}\tag{#1}
    \begin{minipage}[c]{0.85\textwidth} % set width to 0.9 x textwidth
    \em                                % switch on emph
  }
  {\end{minipage}\end{equation}\ignorespacesafterend}

\begin{document}

\title{Free submonoids of hyperbolic monoids}
\author{Matthias Hamann}
\address{Matthias Hamann, Department of Mathematics, University of Hamburg, Hamburg, Germany}
\date{}

\begin{abstract}
In this paper, we prove that infinite cancellative finitely generated hyperbolic monoids never contain $\N\times\N$ as a submonoid but that they contain an element of infinite order and, if they are elementary, then they also contain a free monoid of rank at least~$2$.
As a corollary we obtain that the latter have exponential growth.
We prove these results by analysing the monoid of self-embeddings of hyperbolic digraphs and proving fixed-point theorems for them.
\end{abstract}

\maketitle

\section{Introduction}\label{sec_Intro}

Gray and Kambites~\cite{GK-HyperbolicDigraphsMonoids} were the first to give a geometric notion of hyperbolicity for digraphs (and monoids) that does not just rely on the hyperbolicity in Gromov's sense~\cite{gromov} of the underlying undirected graph but take the direction into account properly.
They were mostly interested in how the class of hyperbolic monoids behaves with respect to various semigroup-theoretic decision problems.
Their notion was caught up in~\cite{H-HyperbolicDigraph}: it was proved that -- under some mild additional assumptions, see Section~\ref{sec_prelim} -- their notion is a \qiy\ invariant.
Here, \qis\ are considered for semimetric spaces, a notion discussed in~\cite{GK-SemimetricSpaces}.
Additionally, a hyperbolic boundary was defined in~\cite{H-HyperDiBound} that is preserved by \qis\ and which has various properties that can be considered as direct analogues of properties for the hyperbolic boundary for hyperbolic spaces.

For hyperbolic metric spaces, not only \qis\ but also isometries play an important role, mostly in order to obtain results for hyperbolic groups.
One important result in that area that can be deduced from properties on isometries of hyperbolic metric spaces is that non-elementary hyperbolic groups contain free subgroups of rank~$2$, see e.\,g.\ Ghys and de la Harpe~\cite[Th\'eor\`eme 8.3.37]{GhHaSur}.
Another one is that $\Z\times\Z$ is never a subgroup of a hyperbolic groups, see e.\,g.\ Ghys and de la Harpe~\cite[Th\'eor\`eme 8.3.34]{GhHaSur}.
We will prove analogue results for those for hyperbolic monoids: we will prove that every infinite cancellative finitely generated hyperbolic monoid contains an element of infinite order (Theorem~\ref{thm_InfOrderElement}), but never $\N\times\N$ as submonoid (Theorem~\ref{thm_NoNtimesN}), and if the monoid is also \emph{non-elementary}, that is, if it has infinitely many hyperbolic boundary points, then it contains $\N\ast\N$ as submonoid (Theorem~\ref{thm_FreeRk2}).
As a corollary of the last theorem, we will obtain that the growth of such monoids is exponential (Theorem~\ref{thm_Growth}).

Instead of prove the above mentioned results directly, we look at a notion that can be seen as an analogue of isometries for digraphs: \se s.
Here, a \emph{\se} of a digraph~$D$ is an injective map $V(D)\to V(D)$ that preserves the adjacency relation.
We ask our digraphs to satisfy some mild assumptions that will be trivially satisfied, when we apply our results to hyperbolic monoids.
One of those is that there exists a base vertex, that is, there is a vertex $o$ of~$D$ with $d(o,x)<\infty$ for all $x\in V(D)$.
While the definition of \se\ is purely based on the digraph, the additional assumptions will imply that \se s induce an injective map from the hyperbolic boundary into that boundary.
Further results for \se s are the following, where a \emph{limit point} is a hyperbolic boundary point $\eta$ such that some sequence $(g_i(o))_{i\in\N}$ for \se s $g_i$ converges to~$\eta$.
(For the definition of the topology and convergence, we refer to Section~\ref{sec_topo}.)
\begin{enumerate}[$\bullet$]
\item If a \se\ $g$ does not fix a finite vertex set setwise -- that is, it is not \emph{elliptic}, then there is a unique hyperbolic boundary point $g^+$ such that any convergent sequence $(g^i(o))_{i\in\N}$ converges to~$g^+$. (Lemma~\ref{lem_directionWellDefined})
\item Every \se\ that is not elliptic fixes only a finite number of hyperbolic boundary points. (Lemmas~\ref{lem_nonEllipticFixFinite1} and~\ref{lem_nonEllipticFixFinite2})
\item If there are at least two limit points in the hyperbolic boundary, then we obtain a result that is similar to the denseness property in the case of isometries of hyperbolic graphs. (Proposition~\ref{prop_DirectionsAlmostDenseInLimit})
\item There are either none, one, two or infinitely many limits points. (Theorem~\ref{thm_numberOfLimits} and Theorem~\ref{thm_freeSubmonoid})
\item If a monoid of \se s neither fixes a finite vertex set setwise or a unique limit point nor has exactly two limit points, then it contains a free monoid of rank~$2$. (Theorem~\ref{thm_fpa})
\end{enumerate}

These results for \se s have natural analogues in the setting of (undirected) trees, see~\cite{H-SelfEmbeddingsTrees}, and they can be seen as an extension of similar results for isometries of hyperbolic spaces or automorphisms of graphs or hyperbolic graphs, see~\cite{GroupsOnMetricSpaces,W-FixedSets}.

The paper is structured as follows.
In Section~\ref{sec_prelim}, we introduce the basic notions that we need for hyperbolic digraphs.
In Section~\ref{sec_se}, we prove the results for \se s of hyperbolic digraphs and we will apply these results in Section~\ref{sec_monoids} to obtain the results for submonoids of monoids.
In Section~\ref{sec_growth}, we will introduce the growth of monoids, as defined in~\cite{GK-SemimetricSpaces}, and prove the theorem on the growth of elementary finitely generated hyperbolic monoids.
In Section~\ref{sec_graphs}, we will carry over an example from Section~\ref{sec_se} to the setting of graphs and thereby obtain a bit knowledge for \se s of graphs as asked in~\cite{H-SelfEmbeddingsTrees}.

\section{Hyperbolic digraphs}\label{sec_prelim}\label{sec_topo}

Let $D$ be a digraph, that is a pair $(V(D),E(D))$ of a vertex set $V(D)$ and an edge set $E(D)$, where the elements of $E(D)$ are pair of vertices $(x,y)$, also denoted by~$xy$.
A \emph{directed path} is a sequence $P=x_0,\ldots,x_n$ of vertices such that $x_ix_{i+1}$ is an edge for all $0\leq i\leq n-1$, we also call it a directed $x_0$-$x_n$ path.
We call $n$ its \emph{length}.
If $i\leq j\leq n$, we denote by $x_iPx_j$ the subpath of~$P$ from~$x_i$ to~$x_j$.
If $i=0$ or $j=n$, we omit that vertex and simply write $Px_j$ or $x_iP$.
A \emph{geodesic} from $x$ to~$y$, or $x$-$y$ \emph{geodesic}, is a shortest directed path from~$x$ to~$y$.
The \emph{distance} $d(x,y)$ from~$x$ to~$y$ is the length of an $x$-$y$ geodesic.
We call $D$ \emph{rooted} if there exists $o\in V(D)$ with $d(o,v)<\infty$ for all $v\in V(D)$ and we call $o$ a \emph{root} of~$D$.
When we say that $D$ is a rooted digraph, we always denote its root by~$o$.

The \emph{in-degree} of a vertex $x$ is the number of its \emph{in-neighbours}, that are the vertices $y$ such that $yx\in E(D)$, and its \emph{out-degree} is the number of its \emph{out-neighbours}, that are the vertices $z$ with $xz\in E(D)$.
The \emph{in-ball} $\cB_k^-(x)$ of radius~$k$ of a vertex $x$ is the set of vertices $y\in V(D)$ with $d(y,x)\leq k$ and the \emph{out-ball} $\cB^+_k(x)$ of radius $k$ of~$x$ is the set of vertices $y\in V(D)$ with $d(x,y)\leq k$.

A \emph{triangle} consists of three vertices of~$D$ and three directed paths, one between every two of those vertices.
We call these paths the \emph{sides} and the three vertices the \emph{end vertices} of the triangle.
The triangle is \emph{geodesic} if all three sides are geodesics.

A geodesic triangle is \emph{$\delta$-thin} for $\delta\geq 0$ if the following holds:
\begin{txteq*}
if $P,Q,R$ are the sides of the triangle and the starting vertex of~$P$ is either the starting or the end vertex of~$Q$ and the end vertex of~$P$ is either the starting or the end vertex of~$R$, then $P$ is contained in $\cB^+_\delta(Q)\cup \cB^-_\delta(R)$.
\end{txteq*}
If all geodesic triangles in~$D$ are $\delta$-thin for some fixed $\delta\geq 0$ then $D$ is \emph{$\delta$-hyperbolic}.
We call $D$ \emph{hyperbolic} if it is $\delta$-hyperbolic for some $\delta\geq 0$.
If not otherwise mentioned, we denote for a hyperbolic digraph the corresponding constant by~$\delta$.

In~\cite{H-HyperDiBound,H-HyperbolicDigraph} it turned out that the following properties are helpful when investigating hyperbolic digraphs.
They basically mean that no finite out- or in-ball contains geodesics of large lengths.
These properties are e.\,g.\ satisfied if the digraphs have bounded out- and bounded in-degree.

\begin{txteqtag}{B1}\label{itm_Bounded1}
There exists a function $\varphi\colon \R \to \R$ such that for every $v \in V(D)$, for every $r \geq 0$ and for all $y, z \in \cB^+_r(v)$ the distance $d(y, z)$ is either $\infty$ or
bounded by~$\varphi(r)$.
\end{txteqtag}
\begin{txteqtag}{B2}\label{itm_Bounded2}
There exists a function $\varphi\colon \R\to\R$ such that for every $v\in V(D)$, for every $r\geq 0$ and for all $y,z\in \cB^-_r(v)$ the distance $d(y,z)$ is either $\infty$ or bounded by~$\varphi(r)$.
\end{txteqtag}

In order to estimate distances, we will need the following results from~\cite{H-HyperbolicDigraph}.

\begin{prop}\label{prop_hyperDi3.3}{\rm \cite[Proposition 3.3]{H-HyperbolicDigraph}}
Let $\delta\geq 0$ and let $D$ be a $\delta$-hyperbolic digraph that satisfies (\ref{itm_Bounded1}) and (\ref{itm_Bounded2}) for the function $\varphi\colon \R\to\R$.
\begin{enumerate}[\rm(i)]
\item If $x,y,z\in V(D)$ are distinct and $P_{u,v}$ is a $u$-$v$ geodesic for all distinct $u,v\in\{x,y,z\}$, then we have
\[
\ell(P_{x,y})\leq (\ell(P_{x,z})+\ell(P_{y,z}))\varphi(\delta+1).
\]
\item If $x, y \in X$ with $d(x, y) \neq \infty$ and $d(y, x) \neq \infty$, then we have
\[
d(x, y) \leq (d(y, x)+1)f(\delta).
\]
\end{enumerate}
\end{prop}

\begin{lem}\label{lem_hyperDi3.4}{\rm \cite[Lemma 3.4]{H-HyperbolicDigraph}}
Let $\delta\geq0$ and let $D$ be a $\delta$-hyperbolic digraph that satisfies (\ref{itm_Bounded1}) and (\ref{itm_Bounded2}) for the function $\varphi\colon\R\to\R$.
Let $x,y,z\in V(D)$, let $P$ be an $x$-$y$ geodesic, $Q$ a $y$-$z$ geodesic and $R$ an $x$-$z$ geodesic.
Then $R$ lies in the out-ball of radius $6\delta+2\delta \varphi(\delta+1)$ around $P\cup Q$ and in the in-ball of the same radius around $P\cup Q$.
\end{lem}

A \emph{(directed) ray} is a one-way infinite directed path with a first vertex and a \emph{(directed) anti-ray} is a one-way infinite directed path with a last vertex.
A \emph{(directed) double ray} is a two-way infinite directed path.
A directed (anti-)ray or directed double ray is \emph{geodesic} is each of its finite directed subpaths is geodesic.

Let $\cR$ be the set of geodesic rays and geodesic anti-rays.
For $R_1,R_2\in\cR$, we write $R_1\leq R_2$ if there exists $m\in\N$ and infinitely many vertices on~$R_1$ with distance at most $M$ to~$R_2$.
This is a quasiorder for hyperbolic digraphs satisfying (B1) and (B2), see \cite[Proposition 4.3]{H-HyperDiBound}.
We call $R_1$ and~$R_2$ \emph{equivalent} if $R_1\leq R_2$ and $R_2\leq R_1$.
This is an equivalence relation whose equivalence classes are the hyperbolic boundary points of~$D$.
We denote by $\rand D$ the \emph{hyperbolic boundary} of~$D$, that is the set of all hyperbolic boundary points of~$D$.

In~\cite{H-HyperDiBound}, it was shown that equivalent geodesic (anti-)rays lie close to each other eventually.
More generally, we have the following.

\begin{lem}\label{lem_hyperDi9.4}{\rm \cite[Corollary 4.4]{H-HyperDiBound}}
Let $D$ be a hyperbolic digraph that satisfies (\ref{itm_Bounded1}) and (\ref{itm_Bounded2}).
If $R_1$ and $R_2$ are geodesic (anti-)rays with $R_1\leq R_2$, then there is a (anti-)subray $R_2'$ of~$R_2$ such that $R_2'\sub \cB^+_{6\delta}(R_1)$.\qed
\end{lem}

A \emph{pseudo-semimetric} on a set~$X$ is a non-negative function $d\colon X\times X\to\R\cup\{\infty\}$ such that
\begin{enumerate}
\item $d(x,x)=0$ for all $x\in X$ and
\item $d(x,y)\leq d(x,z)+d(z,y)$ for all $x,y,z\in X$.
\end{enumerate}

Let us assume from now on that $D$ is rooted.
Then \cite[Theorem 8.2]{H-HyperDiBound} says that there exists a pseudo-semimetric $d_h$ on $D\cup\rand D$ and, moreover, we can assume that it is a \emph{visual} pseudo-semimetric (for a parameter $a>1$), that is, there exists $C > 0$ such that
\[
\frac{1}{C} a^{-\rho(\eta,\mu)} \leq d_h(\eta, \mu) \leq Ca^{-\rho(\eta,\mu)}
\]
for all $\eta,\mu\in D\cup\rand D$, where $\rho(\eta,\mu)$ is defined as follows.
We call a sequence $(x_i)_{i\in\N}$ in $V(D)$ \emph{admissable} for $\eta$ if it is the constant sequence with $x_i=\eta$, in case $\eta\in V(D)$, or if $x_0x_1\ldots$ is a geodesic ray in~$\eta$, otherwise.
Now we let $\rho(\eta,\mu)$ be the supremum with respect of all admissable sequences $(x_i)_{i\in\N}$ for~$\eta$ and all admissable sequences $(y_i)_{i\in\N}$ for~$\mu$ over
\[
\liminf \{d(o,P)\mid i,j\to\infty, P \text{ is an }x_i\text{-}y_j\text{ geodesic}\}.
\]

Pseudo-semimetrics come along with two natural topologies.
In this paper, we only consider the one defined by the open in-balls with respect to~$d_h$, that is, defined by the sets
\[
\{x\in V(D)\cup \rand D\mid d_h(x,v)\leq r\}
\]
for all $r\geq 0$ and $v\in V(D)\cup\rand D$.
In particular, when we speak of convergence, then we mean it with respect to this topology.

We call $D$ \emph{sequentially compact} if for every sequence $(x_i)_{i\in\N}$ in~$V(D)$ that satisfies $d(x_i, x_j) < \infty$ for all $i < j$ has a convergent subsequence.
By \cite[Theorem 9.4]{H-HyperDiBound}, hyperbolic digraphs of bounded in- and out-degree are sequentially compact.

\begin{prop}\label{prop_hyperDi14.1}{\rm \cite[Proposition 9.1\,(iii)]{H-HyperDiBound}}
Let $D$ be rooted hyperbolic digraph that satisfies (\ref{itm_Bounded1}) and (\ref{itm_Bounded2}).
Let $\eta,\mu\in V(D)\cup\rand D$.
If $d_h(\eta, \mu) = 0$ and $\eta\neq\mu$, then either $\eta$ contains only rays and $\mu$ contains only anti-rays or $\eta$ contains only anti-rays and $\mu$ contains only rays.
\end{prop}

\begin{prop}\label{prop_forCorOf14.1}
Let $D$ be a rooted hyperbolic digraph that satisfies (B1) and (B2) for a function $\varphi\colon \R\to\R$.
For every $\eta\in\rand D$ there exists $\mu\in\rand D$ with $d_h(\mu,\eta)=0$ and such that $\mu$ contains rays.
\end{prop}

\begin{proof}
Let us suppose that $\eta$ contains an anti-ray $R=\ldots x_1x_0$.
Let $P$ be an $o$-$x_o$ geodesic and, for every $i\in\N$, let $Q_i$ be an $o$-$x_i$ geodesic.
Considering the geodesic triangle with end vertices $o$, $x_0$ and~$x_i$ and sides $P$, $Q_i$ and $x_iRx_0$, we obtain by Proposition~\ref{prop_hyperDi3.3} that all but the first $d((o,x_0)+\delta)\varphi(\delta+1)$ vertices of~$Q_i$ lie in $\cB^-(x_iRx_o)$.
Thus, if~$Q$ is a geodesic ray defined by the~$Q_i$, that is, for every finite starting subpath, there are infinitely many $Q_i$ that contain it, then we have $Q\leq R$.
\end{proof}

As a corollary of Propositions~\ref{prop_hyperDi14.1} and~\ref{prop_forCorOf14.1}, we obtain the following result.

\begin{cor}\label{cor_corOf14.1}
Let $D$ be a rooted hyperbolic digraph that satisfies (B1) and (B2) for a function $\varphi\colon \R\to\R$.
If there are $\eta\neq\mu\in\rand D$ with $d_h(\eta,\mu)=0$, then $\eta$ contains only rays and $\mu$ contains only anti-rays.\qed
\end{cor}

Let us now prove that we can bound the number of hyperbolic boundary points of distance~$0$ from (or to) a fixed hyperbolic boundary point.

\begin{prop}\label{prop_bpOfDist0}
Let $D$ be a locally finite $\delta$-hyperbolic digraph for some $\delta\geq 0$ that satisfies (\ref{itm_Bounded1}) and (\ref{itm_Bounded2}) for a function $\varphi\colon \R\to\R$.
\begin{enumerate}[\rm (1)]
\item\label{itm_bpOfDist0_1} If $D$ has bounded out-degree, then there exists $K\in\N$ such that for every $\eta\in\rand D$ the set $\{\mu\in\rand D\mid d(\eta,\mu)=0\}$ has size at most~$K$.
\item\label{itm_bpOfDist0_2} If $D$ has bounded in-degree, then there exists $K\in\N$ such that for every $\eta\in\rand D$ the set $\{\mu\in\rand D\mid d(\mu,\eta)=0\}$ has size at most~$K$.
\end{enumerate}
\end{prop}

\begin{proof}
It suffices to prove (\ref{itm_bpOfDist0_1}) since the proof of (\ref{itm_bpOfDist0_2}) is completely analogous.

Let $N$ be the maximum size of out-balls of radius $6\delta$ around vertices of~$D$.
Let us suppose that $X:=\{\mu\in\rand D\mid d(\eta,\mu)=0, \eta\neq\mu\}$ contains more than $N$ elements.
Let $\mu_1,\ldots,\mu_m$ be pairwise distinct elements of~$X$ and $R_i\in\mu_i$ for $1\leq i\leq m$ such that all $R_i$ are pairwise disjoint.
Note that all $R_i$ are anti-rays by Corollary~\ref{cor_corOf14.1}.
Let $\mu_{m+1}\in X$ be distinct from $\mu_1,\ldots,\mu_m$ and let $R\in\mu_{m+1}$.
Then there exists a subanti-ray $R_{m+1}$ of~$R$ that is disjoint from all $R_1,\ldots,R_m$ since it is equivalent to every anti-ray it meets infinitely many times.
Thus, we obtain families $(\mu_i)_{i\in\N}$ of more than $N$ distinct elements of~$X$ and $(R_i)_{i\in\N}$ of pairwise disjoint anti-rays with $R_i\in\mu_i$.

Let $Q\in\eta$.
All the anti-rays $R_i$ lie eventually within the $6\delta$-out-ball of every vertex of~$R$ by \cite[Corollary 4.4]{H-HyperDiBound}.
This contradicts our choice of~$N$.
\end{proof}

While in general two distinct elements of the hyperbolic boundary need not have disjoint open neighbourhood, it becomes true in rooted digraphs if we ask the two hyperbolic boundary points to contain rays, as we will see now.

\begin{lem}\label{lem_disjointBNeighbourhoods}
Let $D$ be a rooted hyperbolic digraph that satisfies (\ref{itm_Bounded1}) and (\ref{itm_Bounded2}) for a function $\varphi\colon \R\to\R$.
Let $\eta,\mu\in\rand D$ be distinct such that both contain rays.
Then there are disjoint open neighbourhoods of~$\eta$ and~$\mu$.
\end{lem}

\begin{proof}
If there are no disjoint open neighbourhoods of $\eta$ and~$\mu$, then every two open neighbourhoods of~$\eta$ and of~$\mu$ have a common vertex, since the topology comes from a visual pseudo-\hm.
Thus, we find a sequence $(x_i)_{i\in\N}$ of vertices in such neighbourhoods that converges to~$\eta$ and to~$\mu$.
Let $R$ be a geodesic ray that is defined by $o$-$x_i$ geodesics, that is, every starting subpath of~$R$ lies in infinitely many of the $o$-$x_i$ geodesics.
Let $\nu\in\rand D$ with $R\in\nu$.
Since $d_h$ is visual, we obtain $d_h(\nu,\eta)=0=d_h(\nu,\mu)$.
Since all three hyperbolic boundary points contain rays, Proposition~\ref{prop_hyperDi14.1} implies $\eta=\nu=\mu$.
This contradiction shows the assertion.
\end{proof}

Let $D_1$ and $D_2$ be digraphs.
Let $\gamma\geq 1$ and $c\geq 0$.
A \emph{$(\gamma,c)$-\qi\ embedding} of $D_1$ into~$D_2$ is a map $f\colon V(D_1)\to V(D_2)$ such that
\[
\frac{1}{\gamma}d_{D_1}(u,v)-c\leq d_{D_2}(f(u),f(v))\leq\gamma d_{D_1}(u,v)+c
\]
for all $u,v\in V(D_1)$.
If additionally for every $w\in V(D_2)$ there exists $v\in V(D_1)$ with $d(f(v),w)\leq c$ and $d(w,f(v))\leq c$, then it is a \emph{$(\gamma,c)$-\qiy} and the digraphs are $(\gamma,c)$-\qi.
If the constants are not important or clear from the context, we simply drop them in those names.
Note that by~\cite[Theorem 12.3]{H-HyperDiBound} every \qiy\ $D_1\to D_2$ induced a homeomorphic embedding $D_1\cup\rand D_1\to D_2\cup\rand D_2$ for hyperbolic digraphs $D_1$ and~$D_2$ that satisfy properties (\ref{itm_Bounded1}) and~(\ref{itm_Bounded2}).

For $\gamma\geq 1$ and $c\geq 0$, a $(\gamma,c)$-\emph{\qg} in a digraph $D$ is a directed path $P$ that satisfies
\[
d_P(u,v)\leq\gamma d_D(u,v)
\]
for all $u,v$ on~$P$ such that $u$ lies between the first vertex of~$P$ and~$v$.
Geodesics and \qg s with the same end vertices lie close to and from each other as the following result from~\cite{H-HyperbolicDigraph} shows.

\begin{lem}\label{lem_HyperDi_7.3}\cite[Corollary 7.2]{H-HyperbolicDigraph}
Let $D$ be a $\delta$-hyperbolic digraph satisfying
(\ref{itm_Bounded1}) and (\ref{itm_Bounded2}) for the function $\varphi\colon \R \to \R$.
Let $\gamma\geq 1$ and $c \geq 0$.
Then there is a constant $\lambda$ depending only on $\delta$, $\gamma$, $c$ and $\varphi$ such that for all $u,v\in V(D)$ every $(\gamma, c)$-quasi-geodesic from $u$ to~$v$ lies in the $\lambda$-out-ball and in the $\lambda$-in-ball of every $u$-$v$ geodesic and vice versa.
\end{lem}

\section{\Se s}\label{sec_se}

Let $D$ be a rooted hyperbolic digraph with bounded in- and out-degrees.
Hence, $D$ satisfies (B1) and (B2).
A \emph{\se} of~$D$ is an injective map $V(D)\to V(D)$ that preserves the adjacency relation, i.\,e.\ it maps edges to edges and non-edges to non-edges.

Our first result is that \se s behave well with respect to the topology on $D\cup\rand D$ under the assumption that the out-degree is not only bounded but constant.
Note that without this additional assumption it is not hard to construct a counterexample to Proposition~\ref{prop_SeInduceInjection}, see Example~\ref{ex_counterToSeInduceInjection}.

\begin{prop}\label{prop_SeInduceInjection}
Let $D$ be a rooted hyperbolic digraph with constant out-degree and bounded in-degree.
Then every \se\ $g$ is a $(1,0)$-\qi\ embedding and induces a homeomorphic embedding
\[
\hat{g}\colon V(D)\cup\rand D \to V(D)\cup\rand D,
\]
that is the union of~$g$ and and injective map $\rand D\to\rand D$.
\end{prop}

\begin{proof}
Since all vertices have the same out-degree, every \se\ must be distance-preserving and hence must be a $(1,0)$-\qi\ embedding.
By \cite[Theorem 7.3]{H-HyperDiBound}, \qis\ induce homeomorphisms between the digraphs with their hyperbolic boundaries.
Since $g$ is a \qiy\ from $D$ to~$g(D)$, we thus obtain a homeomorphism
\[
\hat{g}\colon D\cup \rand D\to g(D)\cup\rand (g(D)).
\]
Constant out-degree implies that there is a canonical map $\rand (g(D)) \to \rand D$, so $\hat{g}$ is a homeomorphic embedding of $V(D)\cup\rand D$ into $V(D)\cup\rand D$.
Obviously, $\hat{g}$ restricted to $\rand D$ is an injective map $\rand D\to\rand$.
\end{proof}

Let us now give an example that shows that the assumption on the out-degree being constant is necessary in for Proposition~\ref{prop_SeInduceInjection}.

\begin{ex}\label{ex_counterToSeInduceInjection}
Let $R=x_1x_2\ldots$ be a ray.
For every $i\geq 1$, let $R_1^i$ and $R_2^i$ be two new rays starting at~$x_i$.
For every $i\geq 2$, we join the $j$-th vertex of $R_1^i$ to the $j$-th vertex of $R_2^i$ by a directed path of length~$2$ and do the same with the roles of $R_1^i$ and $R_2^i$ being swapped.
Thus, for all $i\geq 2$, the rays $R_1^i$ and $R_2^i$ are equivalent, but for $i=1$ they are not.
Since mapping $x_i$ to $x_{i+1}$ defines a \se\ in a canonical way and this \se\ maps $R_1^1$ and $R_2^1$, which are not equivalent, to the equivalent rays $R_1^2$ and $R_2^2$, this \se\ neither is a \qi\ nor extends to a homeomorphism regarding the hyperbolic boundary.
\end{ex}

From now on we will assume that \se s are $(1,0)$-\qis\ without referring to Proposition~\ref{prop_SeInduceInjection} each time.

A \se\ of a hyperbolic digraph is \emph{elliptic} if it fixes a finite vertex set setwise.
Note that all orbits of every elliptic \se\ $g$ in~$D$ must be finite and that all \se s $g^n$ for $n\in\N$ with $n\geq 1$ are elliptic, too.

\begin{prop}\label{prop_nonEllipticQIOfN}
Let $D$ be a rooted hyperbolic digraph with constant out-degree and bounded in-degree.
If $g$ is a \se\ of~$D$ that is not elliptic and if $v\in V(D)$ such that $d(v,g(v))<\infty$, then
\[
\sigma\colon\N\to V(D),\ i\mapsto g^i(v)
\]
is a \qi\ embedding.
\end{prop}

\begin{proof}
Let $R>0$ and set
\begin{align*}
\lambda&:=(2\delta+1)\varphi(\delta+1)+2\delta,\\
\kappa&:=2\lambda \varphi(\delta+1),
\end{align*}
where $\varphi$ is a function that verifies that $D$ has the properties (\ref{itm_Bounded1}) and (\ref{itm_Bounded2}), which exists as $D$ has bounded degree.
Let $k\in\N$ such that for $M:=d(v,g^k(v))$ we have
\[
M>4R+6\delta+2\kappa+4+2(\delta+(3\delta+1+R)\varphi(\delta+1))\varphi(\delta+1)
\]
and let $P$ be a $v$-$g^k(v)$ geodesic.
Let $y$ be on~$P$ with $|d(v,y)-d(y,g^k(v))|\leq 1$, that is a \emph{central} vertex of~$P$.
Let $Q$ be the directed subpath of~$P$ of length $r:=2R+2\kappa$ with $y$ as a central vertex.
Let us first show the following.
\begin{txteq}\label{itm_nonEllipticQIOfN_1}
If $p\in\cB^+_R(v)$ and $q\in\cB^-_R(g^k(v))$ with $d(p,q)=d(v,g^k(v))$ and if $m_1$ lies on a $p$-$q$ geodesic $P'$ with $d(v,y)=d(p,m_1)$, then $m_1$ lies in $\cB^+_\lambda(Q)\cup\cB^-_\lambda(Q)$.
\end{txteq}
In order to prove (\ref{itm_nonEllipticQIOfN_1}), we first show the existence of $m_2,m_3$ on~$P$ with $d(m_1,m_2)\leq\lambda$ and $d(m_3,m_1)\leq\lambda$ and then we show that either $m_2$ or~$m_3$ lies on~$Q$.
Let $Q_1$ be a $v$-$p$ geodesic, $Q_2$ a $q$-$g^k(v)$ geodesic and $Q_3$ a $v$-$q$ geodesic.
We consider the geodesic triangle with end vertices $v$, $p$ and $q$ and sides $Q_1$, $P'$ and $Q_3$.
By hyperbolicity, there is a vertex $m_2'$ on~$Q_3$ with $d(m_1,m_2')\leq\delta$: otherwise, there would be a vertex $x$ on~$Q_1$ with $d(x,m_1)\leq\delta$, which would imply
\begin{align*}
d(v,g^k(v))&\leq d(v,x)+d(x,m_1)+d(m_1,q)+d(q,g^k(v))\\
&\leq M/2+1+\delta+2R
\end{align*}
which is a contradiction to the choice of~$M$.
Now we use a geodesic triangle with end vertices $v$, $q$ and $g^k(v)$ and sides $Q_3$, $Q_2$ and $P$.
Let $u$ be on~$P$ such that it is the last vertex for which there is a vertex $u'$ on $vQ_3m_2'$ with $d(u',u)\leq\delta$ and let $w$ be the out-neighbour of~$u$ on~$P$.
If the distance from~$w$ to~$Q_2$ is at most~$\delta$, then
\begin{align*}
d(v,g^k(v))&\leq d(v,u')+d(u',u)+d(u,g^k(v))\\
&\leq d(v,m_2')+\delta+(1+\delta+R)\\
&\leq d(v,p)+d(p,m_1)+d(m_1,m_2')+2\delta+R+1\\
&\leq M/2+3\delta+2R+2.
\end{align*}
This is a contradiction to the choice of~$M$ and hence there is a vertex $w'$ on $m_2'Q_3q$ with $d(w',w)\leq\delta$.
By Proposition~\ref{prop_hyperDi3.3}, we have
\[
d(u',w') \leq (2\delta+1)\varphi(\delta+1).
\]
Since $m_2'$ lies on $u'Q_3w'$, we obtain
\begin{align*}
d(m_1,w) & \leq d(m_1,m_2')+d(m_2',w')+d(w',w)\\
& \leq 2\delta+(2\delta+1)\varphi(\delta+1).
\end{align*}
This shows the existence of~$m_2$ on~$P$ with $d(m_1,m_2)\leq\lambda$.

Now let $u$ be the last vertex on~$Q_3$ that lies in the out-ball of radius $\delta$ around~$Q_1$.
Then its out-neighbour $w$ must have a vertex $w'$ on~$P'$ with $d(w,w')\leq\delta$.
If $w'$ lies on $m_1P'q$, then  we obtain
\begin{align*}
d(v,g^k(v))&\leq d(v,u)+d(u,w)+d(w,w')+d(w',q)+d(q,g^k(v))\\
&\leq 2R+2\delta +M/2+2,
\end{align*}
which contradicts the choice of~$M$.
Thus, $w'$ lies on $pP'm_1$.
Let $u_1$ be the last vertex on~$Q_3$ that has a vertex $u_1'$ on $pP'm_1$ with $d(u_1,u_1')\leq\delta$ and let $u_2$ be its out-neighbour.
Then there must be $u_2'$ on $m_1P'q$ with $d(u_2,u_2')\leq\delta$.
Applying Proposition~\ref{prop_hyperDi3.3}, we obtain
\[
d(u_1',u_2')\leq (2\delta+1)\varphi(\delta+1).
\]
Thus, we have
\[
d(u_1,m_1)\leq \delta+(2\delta+1)\varphi(\delta+1).
\]

If there were a vertex $x$ on~$Q_2$ with $d(u_1,x)\leq\delta$, then we would have
\[
d(u_1,q)\leq (\delta+R)\varphi(\delta+1)
\]
by Proposition~\ref{prop_hyperDi3.3}.
Applying the same proposition, this would imply
\begin{align*}
d(m_1,q)&\leq (d(u_1,m_1)+d(u_1,q))\varphi(\delta+1)\\
&\leq (\delta+(3\delta+1+R)\varphi(\delta+1))\varphi(\delta+1),
\end{align*}
which contradicts the choice of~$M$.
Thus, there exists $m_3$ on~$P$ with $d(m_3,u_1)\leq\delta$ and hence $d(m_3,m_1)\leq \lambda$.

If $m_2$ lies on $yPg^k(v)$, then we have
\begin{align*}
d(v,m_2)&\leq d(v,p)+d(p,m_1)+d(m_1,m_2)\\
&\leq d(v,y)+R+\lambda
\end{align*}
and, if $m_3$ lies on $vPy$, then we have
\begin{align*}
d(m_3,g^k(v))&\leq d(m_3,m_1)+d(m_1,q)+d(q,g^k(v))\\
&\leq d(y,g^k(v))+R+\lambda.
\end{align*}
As $\lambda\leq\kappa$, the first case implies that $m_2$ lies on~$Q$ and the second case implies that $m_3$ lies on~$Q$.
So let us assume that $m_2$ lies on $vPy$ and $m_3$ lies on $yPg^k(v)$.
Hence, $m_2$ lies on $vPm_3$ and we have $d(m_2,m_3)<\infty$.
Since $d(m_3,m_1)\leq \lambda$ and $d(m_1,m_2)\leq \lambda$, Proposition~\ref{prop_hyperDi3.3} implies
\[
d(m_2,m_3)\leq 2\lambda \varphi(\delta+1)=\kappa.
\]
Since $y$ lies between $m_2$ and~$m_3$, this implies that either $m_2$ or~$m_3$ must lie on~$Q$.
This finishes the proof of~(\ref{itm_nonEllipticQIOfN_1}).

\medskip

Since $D$ has bounded in- and out-degree, there exists $N\in\N$ such that, for every $u\in V(D)$, the set $\cB^+_\lambda(u)\cup \cB^-_\lambda(u)$ contains at most~$N$ vertices.
Then there are at most $2N(R+\kappa)=rN$ vertices that lie in $\cB^+_\lambda(V(Q))\cup\cB^-_\lambda(V(Q))$.

Since $g$ is not elliptic, we obtain that all vertices $g^i(y)$ that lie on the images $g^i(P)$ for $0\leq i\leq rN$ must be distinct.
Hence we find $g^{\varrho(R)}(y)$ with $1\leq\varrho(R)\leq rN$ such that $g^{\varrho(R)}(y)$ lies outside of $\cB^+_\lambda(V(Q))\cup\cB^-_\lambda(V(Q))$.
So we have either $g^{\varrho(R)}(v)\notin\cB^+_R(v)$ or $g^{\varrho(R)}(g^k(v))\notin\cB^-_R(g^k(v))$.
Since $g$ is a \se\ and a $(1,0)$-\qiy, the second case implies $g^{\varrho(R)}(v)\notin\cB^+_R(v)$, too.

We will now show that
\begin{txteq}\label{itm_nonEllipticQIOfN_2}
$d(v,g^{rN}(v))\geq R$.
\end{txteq}
Let us suppose that $d(v,g^{rN}(v))<R$.
Let $r'\in\N$ with $0\leq r'<rN$ and let $n\in\N$.
Set $m:=nrN+r'$ and let $\varepsilon>0$ such that $d(v,g^{rN}(v))<R-\varepsilon$.
Then we have
\begin{align*}
d(v,g^m(v))&\leq d(v,g^{nrN}(v))+d(g^{nrN}(v),g^m(v))\\
&\leq nd(v,g^{rN}(v))+d(v,g^{r'}(v))\\
&<n(R-\varepsilon)+d(v,g^{r'}(v)).
\end{align*}
There exists $N_0\in\N$ such that, if $n\geq N_0$, we have $n\varepsilon>d(v,g^{r'}(v))$ and hence
\[
n(R-\varepsilon)+d(v,g^{r'}(v))<nR.
\]
Let $R''\in\N$ and set $r'':=2(R''+\kappa)$ such that $R<R''$ and $\varrho(R'')>rN$ and such that for that $n_0\in\N$ with
\[
n_0rN\leq\varrho(R'')<(n_0+1)rN
\]
we have $n_0\geq N_0$.
By the above inequalities with $\varrho(R'')$ instead of~$m$, we have
\[
d(v,g^{\varrho(R'')}(v))<n_0R\leq n_0r\leq\varrho(R'')/N\leq r''.
\]
But as we have observed above, we also have $d(v,g^{\varrho(r'')}(v))\geq r''$.
This contradiction shows~(\ref{itm_nonEllipticQIOfN_2}).

\medskip

Let $P'$ be a geodesic from~$v$ to~$g(v)$.
Let us now prove that $P^\infty:=\bigcup_{n\in\N}g^{n}(P')$ is a \qg.
Let $T$ be an $a$-$b$ geodesic for $a,b\in V(P^\infty)$ such that $P^\infty$ passes through $a$ before it passes through~$b$.
There are $n_a, n_b\in\N$ such that for $N_a:=2(n_a+\kappa)$ the vertex $a$ has distance at most $2(1+\kappa) N d(v,g(v))$ to $g^{N_aN}(v)$ and for $N_b:=2(n_b+\kappa)$ the vertex $b$ has distance at most $2(1+\kappa) Nd(v,g(v))$ from $g^{N_bN}(v)$.
If $n_b<n_a$, then $n_b=n_a-1$ and $d(a,b)\leq 2(1+\kappa)N d(v,g(v))$.
So let us assume that $n_a\leq n_b$.
By~(\ref{itm_nonEllipticQIOfN_2}), we have
\begin{align*}
d(a,b)&\geq d(g^{N_aN}(v),g^{N_bN}(v)) - d(a,g^{N_aN}(v)) - d(g^{N_bN}(v),b)\\
&\geq (n_b-n_a) - 4(1+\kappa)Nd(v,g(v)).
\end{align*}
Thus, we obtain
\begin{align*}
d_{P^\infty}(a,b)&\leq d(a,g^{N_aN}(v)) + d(g^{N_aN}(v),g^{N_bN}(v)) + d(g^{N_bN}(v),b)\\
&\leq (N_b-N_a)Nd(v,g(v))+4(1+\kappa)Nd(v,g(v))\\
&= 2(n_b-n_a)Nd(v,g(v))+4(1+\kappa)Nd(v,g(v))\\
&\leq 2Nd(v,g(v)) d(a,b)+4(1+\kappa)Nd(v,g(v))(1+2Nd(v,g(v))).
\end{align*}
This shows that $P^\infty$ is a $(\gamma,c)$-\qg\ for $\gamma=2Nd(v,g(v))$ and
\[
c=4(1+\kappa)Nd(v,g(v))(1+2Nd(v,g(v))).
\]
Thus, $\sigma$ is a $(\gamma,c)$-\qi\ embedding.
\end{proof}

Let $g$ be a non-elliptic \se\ of~$D$, where $D$ satisfies the assumptions of Proposition~\ref{prop_nonEllipticQIOfN}.
By Proposition~\ref{prop_nonEllipticQIOfN}, the sequence $(g^n(o))_{n\in\N}$ gives rise to a \qg\ ray~$Q$.
Also, the geodesics from~$o$ to~$g^n(o)$ define a geodesic ray~$R$ starting at~$o$ in that infinitely many of these geodesics share a first common edge among which infinitely many share a second common edge and so on.
By Lemma~\ref{lem_HyperDi_7.3}, applied to the directed subpaths of~$Q$ and the geodesics from~$o$ to~$g^n(o)$, we obtain that there exists $\lambda$, depending only on~$\delta$, $\varphi$ and the constants for the \qg, such that $Q$ lies in the $\lambda$-out- and $\lambda$-in-ball of~$R$ and vice versa.
While $R$ need not be uniquely determined, this argumentation shows that all possible geodesic rays obtained in the way that we obtained~$R$ are equivalent and hence lie in a uniquely determined hyperbolic boundary point, denoted by~$g^+$, which we call the \emph{direction} of~$g$.
Note that~$Q$ converges to~$g^+$, which can be seen easily using that the pseudo-semimetric that defined our topology is visual.

\begin{lem}\label{lem_directionWellDefined}
Let $D$ be a rooted hyperbolic digraph with constant out-degree and bounded in-degree.
If $g$ is a \se\ of~$D$ that is not elliptic, then there exists no $v\in V(D)$ and $\eta\in\rand D$ with $d_h(g^+,\eta)\neq 0$ such that $(g^n(v))_{n\in\N}$ converges to~$\eta$.
\end{lem}

\begin{proof}
Let $v\in V(D)$ and $\eta\in\rand D$ such that $(g^n(v))_{n\in\N}$ converges to~$\eta$.
Then $d(o,v)=:k<\infty$ and hence $d_h(g^n(o),g^n(v))=k$.
Thus, $(g^n(o))_{n\in\N}$ converges to~$\eta$ as the pseudo-semimetric is visual.
This implies $d_h(g^+,\eta)=0$.
\end{proof}

Note that it follows immediately from the definition of~$g^+$ that $g$ fixes its direction, i.\,e.\ $g(g^+)=g^+$.
In the next results, we will prove that there are only finitely many hyperbolic boundary points that are fixed by~$g$ (Lemmas~\ref{lem_nonEllipticFixFinite1} and~\ref{lem_nonEllipticFixFinite2}).
They occur in two results since we deal with the two cases whether the hyperbolic boundary point contains rays or anti-rays separately.

\begin{lem}\label{lem_nonEllipticFixFinite1}\label{lem_2BPfixed-a}
Let $D$ be a rooted hyperbolic digraph with constant out-degree and bounded in-degree.
If $g$ is a \se\ that is not elliptic, then there exist only finitely many hyperbolic boundary points of~$D$ that are fixed by~$g$ and contain rays.

More specifically, the number of hyperbolic boundary points fixed by~$g$ that contain rays is bounded by
\[
|\{\omega\in\rand D\mid d_h(g^+,\omega)=0,\ \omega\text{ contains anti-rays}\}|+1.
\]

Furthermore, if $g$ fixes a hyperbolic boundary point other than~$g^+$ that contains rays, then there exists $\mu\in\rand D$ with $d_h(g^+,\mu)=0$ that contains anti-rays.
\end{lem}

\begin{proof}
For this proof, let $\varphi\colon \R\to\R$ be a function such that $D$ satisfies (\ref{itm_Bounded1}) and (\ref{itm_Bounded2}) for this function.
Let $\eta\in\rand D$ with $\eta\neq g^+$ be fixed by~$g$ and let us assume that $\eta$ contains rays.
Note that we have $d_h(g^+,\eta)>0$ by Proposition~\ref{prop_hyperDi14.1}.
Since $o$ is a root of~$D$, we have $d_h(o,\eta)<\infty$ and thus also $d_h(g^n(o),\eta)<\infty$ for all $n\in\N$.
Since $d_h(g^+,\eta)>0$, there exists a geodesic $o$-$\eta$ ray $R_\eta$ by \cite[Proposition 9.3]{H-HyperDiBound}.
Since $g^n(o)$ converges to~$g^+$ for $n\to\infty$ and since the pseudo-semimetric $d_h$ is visual, there exists $x$ on~$R_\eta$ and $N\in\N$ such that for all $n\geq N$ we have $d(g^n(o),x)<\infty$.
Note that there exists $L>0$, depending only on $d(o,x)$, such that all but the last $L$ vertices of each $g^n(o)$-$x$ geodesic lie within the $\delta$-out-ball of each $o$-$g^n(o)$ geodesic by the hyperbolicity condition on the geodesic triangles with end vertices $o$, $g^n(o)$ and~$x$.
This, applied for all~$n\geq N$, shows that the sides from $g^n(o)$ to~$x$ of these geodesic triangles define a geodesic anti-ray $Q$ just the same way, the geodesics from~$o$ to~$g^n(o)$ define a geodesic ray in~$g^+$, that is, we find an increasing sequence $(n_i)_{i\in\N}$ of~$\N$ such that all $g^{n_i}$-$x$ geodesics share their last edge, all but at most $n_1$ share their second to last edge and so on.
By the property on the sides of the geodesic triangles with end vertices $o$, $g^n(o)$ and~$x$, we obtain that all but the last $L$ vertices of~$Q$ lie in the $\delta$-out-ball of some $o$-$g^+$ geodesic.
Thus, $Q$ lies in a hyperbolic boundary point $\eta^-$ with $d_h(g^+,\eta^-)=0$.
Since $d_h(Q,\eta)<\infty$, we have $d_h(\eta^-,\eta)<\infty$.
If $d_h(\eta^-,\eta)=0$, then we have $\eta=\eta^-$ and hence $\eta=g^+$ by Proposition~\ref{prop_hyperDi14.1}, a contradiction to the choice of~$\eta$.
Thus, we have $d_h(\eta^-,\eta)>0$ and hence \cite[Proposition 9.3]{H-HyperDiBound} implies that there is a geodesic $\eta^-$-$\eta$ double ray~$P$.
Let us now show the following.

\begin{txteq}\label{itm_nonEllipticFixFinite1_1}
For all $n\in\N$, the geodesic double ray $g^n(P)$ is an $\eta^-$-$\eta$ double ray. In particular, $g$ fixes~$n^-$.
\end{txteq}

Note that $g^n(Q)$ can be defined the same way as~$Q$ just with $g^{n_i+n}(o)$-$g^n(x)$ geodesics instead of $g^{n_i}(o)$-$x$ geodesics.
Since $R_\eta$ and $g(R_\eta)$ lie in~$\eta$, they are equivalent.
There exists $y$ on~$R_\eta$ with $d(x,y)<\infty$ and $d(g^n(o),y)<\infty$.
Let $P_1$ be a $g^{n_i}(o)$-$x$ geodesic, $P_2$ a $g^{n_i}(o)$-$y$ geodesic, $P_3$ a $g^{n_i+n}(o)$-$y$ geodesic, $P_4$ a $g^{n_i+n}(o)$-$g^n(x)$ geodesic, $P_5$ a $g^{n_i}(o)$-$g^{n_i+n}(o)$ geodesic, $P_6$ an $x$-$y$ geodesic and $P_7$ a $g^n(x)$-$y$ geodesic.
Applying hyperbolicity to the geodesic triangle with sides $P_1$, $P_2$ and~$P_6$ and using Proposition ~\ref{prop_hyperDi3.3}, we obtain that all but at most the last $(d(x,y)+\delta)\varphi(\delta+1)$ vertices of~$P_2$ lie in $\cB^+_\delta(P_1)$.

Now we consider the geodesic triangle with sides $P_2$, $P_3$ and~$P_5$.
Again, at most the first $(d(g^{n_i}(o),g^{n_i+n}(o))+\delta)\varphi(\delta+1)$ vertices of~$P_2$ lies in $\cB^+_\delta(P_5)$ by Proposition ~\ref{prop_hyperDi3.3} and all others lie in $B^-_\delta(P_3)$.
If $z$ is the first vertex on~$P_2$ for which $zP_2$ lies in $\cB^-_\delta(P_3)$, then let $z'$ be on~$P_3$ with $d(z,z')\leq\delta$.
By Proposition~\ref{prop_hyperDi3.3}, then we have
\begin{align*}
d(g^{n_i+n}(o),z')\leq\ & (d(g^{n_i}(o),g^{n_i+n}(o))+d(g^{n_i}(o),z'))\varphi(\delta+1)\\
\leq\ & (2 d(g^{n_i}(o),g^{n_i+n}(o))+d(g^{n_i}(o),z)+d(z,z'))\varphi(\delta+1)\\
\leq\ & (2 d(g^{n_i}(o),g^{n_i+n}(o)) + 2\delta+1)\varphi(\delta+1)\\
=:&\ \kappa.
\end{align*}
If $uv$ is an edge on $P_2$ and $u',v'$ on~$P_3$ with $d(u,u')\leq\delta$ and $d(v,v')\leq\delta$, then Proposition ~\ref{prop_hyperDi3.3} implies $d(u',v')\leq (2\delta+1)\varphi(\delta+1)$.
So all but at most the first $\kappa$ vertices of~$P_3$ have distance at most $\delta+(2\delta+1)\varphi(\delta+1)$ from~$P_2$.
Also at most the last $(2\delta+1)\varphi(\delta+1)(d(x,y)+\delta)\varphi(\delta+1)$ vertices do not lie in $\cB^+_{\lambda}(P_1)$ for
\[
\lambda:=2\delta+(2\delta+1)\varphi(\delta+1).
\]

Considering the geodesic triangle with sides $P_3$, $P_4$ and~$P_7$, we obtain that all but the last $(d(g^n(x),y)+\delta)\varphi(\delta+1)$ vertices of~$P_4$ lie in $\cB^+_\delta(P_3)$.
We apply Proposition~\ref{prop_hyperDi3.3} again to conclude that there exists constants $\lambda_1$ and~$\lambda_2$ that depend only on~$\delta$, $d(x,y)$, $d(g^n(x),y)$ and $d(g^{n_i}(o),g^{n_i+n}(o))$ such that all but the first $\lambda_1$ and the last $\lambda_2$ vertices of~$P_4$ lie in the out-ball of radius $3\delta+(2\delta+1)\varphi(\delta+1)$ around~$P_1$.

We note that the orientation of the side $P_5$ did not play any role in the arguments.
Thus, we also obtain the symmetric result, that is, that all but the first $\lambda_1$ and the last $\lambda_2$ vertices of~$P_1$ lies in the out-ball of radius $3\delta+(2\delta+1)\varphi(\delta+1)$ around~$P_4$.

By the choice of~$Q$ and~$g^n(Q)$, we obtain that they are equivalent and hence lie in the same hyperbolic boundary point.
This shows that $g^n$ fixes~$\eta^-$, which immediately implies~(\ref{itm_nonEllipticFixFinite1_1}).

\medskip

Let $R\in g^+$.
Since $d_h(g^+,\eta^-)=0$, there is an anti-subray $P'$ of~$P$ that lies completely within the $6\delta$-out-ball of~$R$ by Lemma~\ref{lem_hyperDi9.4}.
Since $g$ fixes $\eta^-$ and~$\eta$, all images $g^n(P)$ lie within the $6\delta$-out-ball of~$P$ and hence that part of $g^n(P)$ that does not lie in $\cB^+_{6\delta}(P-P')$ lies within the $12\delta$-out-ball of~$R$.
For all but finitely many vertices $u$ of~$R$, there exists an edge $vw$ on~$g^n(P)$ such that $\cB^-_{12\delta}(v)\cap V(Ru)\neq\es$ and $\cB^-_{12\delta}(w)\cap V(uR)\neq\es$.
Let $v'\in V(Ru)$ and $w'\in V(uR)$ such that $d(v',v)\leq 12\delta$ and $d(w',w)\leq 12\delta$.
Then Proposition~\ref{prop_hyperDi3.3} implies
\[
d(v',w')\leq (24\delta+1) \varphi(\delta+1)
\]
and hence
\[
d(u,g^n(P))\leq (24\delta+1) \varphi(\delta+1)+12\delta=:K.
\]
Note that the number of vertices in any out-ball of radius~$K$ is bounded by $\sum_{i=0}^K\Delta^i$, where $\Delta$ denotes the maximum out-degree of~$D$.
If $\eta_1\neq\eta_2$ are fixed by~$g$, then the geodesic $\eta_1^-$-$\eta_1$ and $\eta_2^-$-$\eta_2$ double rays have to be disjoint eventually.
Thus, we can map the disjoint parts by some $g^n$, for $n$ large enough, into the $K$-out-ball of a vertex $x$ on~$R$.
Since their images under $g^n$ have to be disjoint as well, there may be at most $\sum_{i=0}^K\Delta^i$ many such choices for~$\eta$.

It remains to prove the bound on the number of fixed hyperbolic boundary points that contain rays.
In order to prove that, is suffices to show that no hyperbolic boundary point $\nu$ with $d_h(g^+,\nu)=0$ sends geodesic double rays to two distinct hyperbolic boundary points that are fixed by~$g$.
Together with Proposition~\ref{prop_bpOfDist0}, this then implies directly the assertion.
So let us suppose for a contradiction that $\eta$ and another hyperbolic boundary point $\mu\neq\eta$ have the property that for the same $\nu\in\rand D$ with $d_h(g^+,\nu)=0$ there is a geodesic $\nu$-$\eta$ double ray~$Q_\eta$ and a geodesic $\nu$-$\mu$ double ray~$Q_\mu$.
Set $G^+:=\{\omega\in\rand D\mid d_h(g^+,\omega)=0\}$.
Note that $G^+$ is finite by Proposition~\ref{prop_bpOfDist0}.
Since $\eta\neq\mu$, we apply Lemma~\ref{lem_hyperDi9.4} to obtain a vertex $x$ on~$Q_\eta$ such that the out-ball $B$ of radius $6\delta$ around~$x$ meets all geodesic double rays from all (finitely many) elements of~$G^+$ to~$\eta$ but none from $G^+$ to~$\mu$.
Since $g$ fixes $g^+$, it also fixes $G^+$ setwise.
Thus, for all $n\in\N$, the ball $g^n(B)$ meets all geodesic double rays from $G^+$ to~$\eta$ but none from~$G^+$ to~$\nu$.
But since the anti-rays in~$Q_\eta$ and $Q_\mu$ are all equivalent, this is a contradiction that shows our claim and hence the assertion.
\end{proof}

\begin{lem}\label{lem_2BPfixed}
Let $D$ be a rooted hyperbolic digraph with constant out-degree and bounded in-degree.
If $g$ is a \se\ of~$D$ that is not elliptic and if there are $\eta,\mu\in\rand D$ with $d_h(\eta,\mu)=0$ and $\eta\neq g^+$ such that $\eta$ contains rays and $\mu$ contains anti-rays and such that either $\eta$ or~$\mu$ is fixed by~$g$, then $\eta=\mu$ and $g^+$ contains anti-rays.
\end{lem}

\begin{proof}
Let $\varphi\colon \R\to\R$ be a function such that $D$ satisfies (\ref{itm_Bounded1}) and (\ref{itm_Bounded2}) for this function.
Let $Q=\ldots x_{-1}x_0$ be a geodesic anti-ray in~$\mu$.
Since $g^+\neq \eta$, we find a geodesic double ray $R$ from some $\nu\in\rand D$ with $d_h(g^+,\nu)=0$ to~$\eta$ as in the proof of Lemma~\ref{lem_nonEllipticFixFinite1}.

Since $d_h(\eta,\mu)=0$, there exists $N\in\N$ with $d(R,x_{-i})\leq 6\delta$ for all $i\geq N$ by Lemma~\ref{lem_hyperDi9.4}.
By replacing $Q$ by a anti-subray, we may assume that $N=0$.
For every $n\in\N$, the image of~$Q$ under $g^n$ lies in the $6\delta$-out-ball of some geodesic double ray from a hyperbolic boundary point of distance $0$ from~$g^+$ either to~$\eta$, if $g$ fixes $\eta$, or to a hyperbolic boundary point of distance $0$ to~$\mu$, if $g$ fixes~$\mu$.
Since there are only finitely many such hyperbolic boundary points by Proposition~\ref{prop_bpOfDist0}, there exists one geodesic double ray $R'$ from a hyperbolic boundary point of distance~$0$ from~$g^+$ either to~$\eta$, if $\eta$ is fixed by~$g$, or to a hyperbolic boundary point of distance~$0$ to~$\mu$ such that $g^i(Q)$ lies in the $6\delta$-out-ball of $R'$ for all $i$ in some infinite subset $I$ of~$\N$.
Since the out-degrees are constant, we can use the construction that we used multiple times so far, to obtain a geodesic double ray $Q'$ that is defined by the images $g^i(Q)$ and goes either from~$\mu$, if $g$ fixes~$\mu$, or from a hyperbolic boundary point $\mu'$ with $d_h(\eta,\mu')=0$ to a hyperbolic boundary point $\nu'$ that has distance $0$ to the hyperbolic boundary point at which $R'$ starts.
In the case that $g$ fixes~$\mu$, we set $\mu':=\mu$.

Let $\eta'$ be the hyperbolic boundary point that contains the subrays of~$R'$.
So we have $\eta=\eta'$ if $g$ fixes~$\eta$.
Let $x$ be a vertex on~$R'$.
Since we have $d_h(\eta',\mu')=0$, we find a directed $R'$-$Q'$ path $P_1$ of length at most $6\delta$ and distance more than $7\delta\varphi(\delta+1)$ from~$x$.
Similarly, we find a directed $Q'$-$R'$ path $P_2$ of length at most $6\delta$ of distance more than $8\delta\varphi(\delta+1)$ to~$x$.
Let $x_i$ and $y_i$ be the starting and end vertices of $P_i$ for $i=1,2$, respectively.
Applying hyperbolicity and Proposition~\ref{prop_hyperDi3.3}, we conclude that all but at most the last $7\delta$ vertices of any $y_1$-$y_2$ geodesic lies within the $\delta$-out-ball of~$Q'$ and all but at most the first $7\delta\varphi(\delta+1)$ vertices and at most the last $8\delta\varphi(\delta+1)$ vertices of $y_2R'x_1$ lies in the $2\delta$-out-ball of~$Q'$.
Thus, $x$ lies within the $2\delta$-out-ball of~$Q'$ and hence $R'$ lies within the $2\delta$-out-ball of~$Q'$.
This implies $d_h(\mu',\eta')=0$ and hence $\eta'=\mu'$.
With a similar argument, we obtain $\nu'=\omega$, where $\omega$ is the hyperbolic boundary point that contains the anti-subrays of~$R'$.
Proposition~\ref{prop_hyperDi14.1} implies $g^+=\omega$.
In particular, $g^+$ contains anti-rays.
\end{proof}

\begin{lem}\label{lem_nonEllipticFixFinite2}
Let $D$ be a rooted hyperbolic digraph with constant out-degree and bounded in-degree.
If $g$ is a \se\ of~$D$ that is not elliptic, then there exist only finitely many hyperbolic boundary points of~$D$ that are fixed by~$g$ and that contain anti-rays.

More precisely, either the only hyperbolic boundary points $\eta$ fixed by~$g$ that contain anti-rays satisfy $d_h(g^+,\eta)=0$ or there is a unique hyperbolic boundary point $g^-$ with $d_h(g^+,g^-)\neq 0$ that contains anti-rays and is fixed by~$g$.
In the latter case, $g^+$ contains anti-rays and $g^-$ contains rays.
\end{lem}

\begin{proof}
We assume that there is a hyperbolic boundary point $g^-$ with $d_h(g^+,g^-)\neq 0$ that contains anti-rays and that is fixed by~$g$.
Let $Q=\ldots x_{-1}x_0$ be a geodesic anti-ray in~$g^-$.
Then there is a ray starting at~$o$ such that its first edge lies on infinitely many $o$-$x_i$ geodesics, its second edge on infinitely many of those geodesics that already shared their first edge and so on.
Let $\eta\in\rand D$ contain that ray.
Then by construction and as $d_h$ is visual, we have $d_h(\eta,g^-)=0$.
Lemma~\ref{lem_2BPfixed} implies $\eta=g^-$ and that $g^+$ contains anti-rays.
Thus, every element of~$\rand D$ that is fixed by $g$ and contains anti-rays, contains rays as well.
So Lemma~\ref{lem_nonEllipticFixFinite1} implies that $g^-$ and~$g^+$ are the only such hyperbolic boundary points.
\end{proof}

The following result is an immediate corollary of Lemmas~\ref{lem_nonEllipticFixFinite1} and~\ref{lem_nonEllipticFixFinite2}.

\begin{cor}\label{cor_nonEllipticFixFinite}
Let $D$ be a rooted hyperbolic digraph with constant out-degree and bounded in-degree.
If $g$ is a \se\ of~$D$ that is not elliptic but fixes a hyperbolic boundary point $\eta$ with $d(g^+,\eta)\neq 0$ that contains anti-rays, then $\eta$ and~$g^+$ are the only hyperbolic boundary points fixed by~$g$.\qed
\end{cor}

Now Corollary~\ref{cor_nonEllipticFixFinite} together with Lemma~\ref{lem_2BPfixed} implies the following.

\begin{cor}\label{cor_nonEllipticFixFinite2}
Let $D$ be a rooted hyperbolic digraph with constant out-degree and bounded in-degree.
If $g$ is a \se\ of~$D$ that is not elliptic but fixes a hyperbolic boundary point $\eta\neq g^+$ such that some $\mu\in\rand D$ with $d_h(\eta,\mu)=0$ contains anti-rays, then there is no hyperbolic boundary point distinct from $\eta$ and~$g^+$ that is fixed by~$g$.

Furthermore, both fixed hyperbolic boundary points contain rays and anti-rays.\qed
\end{cor}

Lemmas~\ref{lem_nonEllipticFixFinite1} and~\ref{lem_nonEllipticFixFinite2} are best possible in that we cannot ask for some constant~$K$ such that for every rooted hyperbolic digraph with constant out-degree and bounded in-degree and for every \se\ $g$ that is not elliptic there are at most~$K$ hyperbolic boundary points that are fixed by~$g$ as the following example shows.

\begin{ex}\label{ex_seMayFixArbitrarilyManyBP}
Let $n\in\N$.
Let $x_0x_1\ldots$ be a ray and for every $1\leq i\leq n$ let $\ldots y_{-1}^iy_0^iy_1^i\ldots$ be a double ray, i.\,e.\ we have edges $x_ix_{i+1}$ for all $i\in\N$ and $y_jy_{j+1}$ for all $j\in\Z$.
Let $D_1$ be the digraph obtained from the disjoint union of these ray and double rays with $x_jy_{-j}^i$ as additional edges for every $j\in\N$ and every $1\leq i\leq n$.
Let $D_2$ be a finite digraph with constant out-degree~$3$.
For every vertex $y_j^i$, we add a copy of $D_2$ to~$D_1$ and join $y_j^i$ to exactly two vertices of that copy.
The resulting digraph $D$ is hyperbolic and its hyperbolic boundary has $2n+1$ elements.
Furthermore, it has constant out-degree and bounded in-degree.
Let $g$ be the \se\ of~$D$ that maps every $x_j$ to $x_{j+1}$, every $y_j^i$ to $y_{j+1}^i$ and the copies of~$D_2$ accordingly.
This \se\ fixes each hyperbolic boundary point but no finite vertex set setwise.
\end{ex}

If $g$ is a \se\ of a rooted hyperbolic digraph $D$ with constant out-degree and bounded in-degree that is not elliptic, then we call it
\begin{enumerate}[$\bullet$]
\item \emph{hyperbolic} if it fixes a hyperbolic boundary point $g^-$ with $d_h(g^+,g^-)>0$ that contains anti-rays;
\item \emph{parabolic} otherwise.
\end{enumerate}

Note that if $g$ is hyperbolic, then Corollary~\ref{cor_nonEllipticFixFinite} implies that $g^+$ and $g^-$ are the only hyperbolic boundary points fixed by~$g$.
If $g$ is parabolic, then it may fix an arbitrary number of hyperbolic boundary all but one of which consist of only rays.

Let us now show that being parabolic or being hyperbolic is preserved by taking powers of the \se.

\begin{lem}\label{lem_PowerHyperPara}
Let $g$ be a \se\ of a rooted hyperbolic digraph $D$ with constant out-degree and bounded in-degree.
\begin{enumerate}[\rm (1)]
\item Then $g$ is hyperbolic if and only if $g^n$ is hyperbolic for all $n\geq 1$; and
\item $g$ is parabolic if and only if $g^n$ is parabolic for all $n\geq 1$.
\end{enumerate}
\end{lem}

\begin{proof}
If $g$ is hyperbolic, then $g^n$ fixes $g^+$ and $g^-$, too, and it is obviously not elliptic.
Since $g^-$ contains anti-rays, $g^n$ must be hyperbolic.

If $g$ is parabolic, no $g^n$ can be elliptic.
Suppose that for some $n\geq 2$, the \se\ $g^n$ is not parabolic.
Then it is hyperbolic and hence fixes a unique hyperbolic boundary $(g^n)^-$ that contains anti-rays and that is distinct from~$g^+$.
By Corollary~\ref{cor_nonEllipticFixFinite}, the \se\ $g$ cannot fix $(g^n)^-$ and thus $g((g^n)^-)$ must be distinct from $(g^n)^-$ and~$g^+$.
Obviously, $g^n$ must also fix $g((g^n)^-)$, a contradiction to $g$ being parabolic.

Trivially, \se s $g$ are elliptic if and only if $g^n$ is elliptic for every $n\in\N$.
Thus, if $g^n$ is hyperbolic, $g$ is neither parabolic nor elliptic and hence it is hyperbolic.
Similarly, if $g^n$ is parabolic, $g$ is neither hyperbolic nor elliptic, so it is parabolic.
\end{proof}

A priori, a hyperbolic boundary point that is fixed by~$g^n$ for a non-elliptic \se\ $g$ need not be fixed by~$g$.
However, the following result shows a situation where we can conclude that $g$ fixes that hyperbolic boundary point.

\begin{lem}\label{lem_PowerFixedDirectionImpliesHyper}
Let $D$ be a rooted hyperbolic digraph with constant out-degree and bounded in-degree.
Let $g$ be \se s of~$D$ that is not elliptic.
If $g^n$ fixes $\mu\in\rand D$ for some $n\in\N$ and if there is a hyperbolic boundary point $\eta$ with $d_h(\mu,\eta)=0$ that contains anti-rays, then $g$ fixes $h^+$.
\end{lem}

\begin{proof}
We may assume that $g^+\neq \mu$ since otherwise the assertion holds trivially.
There are only finitely many hyperbolic boundary points $\eta$ with $d_h(\mu,\eta)=0$ according to Proposition~\ref{prop_bpOfDist0}.
So there is some $m\in\N$ such that $g^{nm}$ fixes one of them.
Thus, Corollary~\ref{cor_nonEllipticFixFinite} shows that $g^{nm}$ is hyperbolic.
Lemma~\ref{lem_PowerHyperPara} implies that $g$ and $g^n$ are hyperbolic, too.
Since $g$ and $g^n$ must fix the same hyperbolic boundary points we conclude that $g$ fixes~$\mu$.
\end{proof}

The set of all self-embeddings of~$D$ forms a monoid.
In the rest of this section, we focus on results for submonoids of that.
Let $M$ be a monoid of \se s on~$D$.
We denote by~$\cD(M)$ the set of directions of all non-elliptic elements of~$M$.
The set
\begin{align*}
\cL(M):= \{&\eta\in\rand D\mid \exists \text{ sequence }(g_i)_{i\in\N}\text{ in } M\colon (g_i(o))_{i\in\N}\text{ converges to }\eta,\\&\eta\text{ contains rays}\}
\end{align*}
is the \emph{limit set} of~$M$, its elements are the \emph{limit points}.
Obviously, $\cD(M)$ is a subset of $\cL(M)$.

Contrary to the case of automorphims on graphs, in our situation $\cD(M)$ is not dense in $\cL(M)$ as the following example shows.

\begin{ex}\label{ex_DirectionsNotDenseInLimit}
Let $I\sub\N$ with $0\notin I$ be such that for all $m\geq 0$ and all $n>0$, we have
\[
\{n+i\mid i\in I, i\geq m\}\neq\{i\in I\mid i\geq m+n\}.
\]
Let $R=x_0x_1\ldots$ be a ray.
To each $x_i$ with $i\in I$ we attach a new ray $x_0^ix_1^i\ldots$ such that $x_0^i=x_i$.
We continue to add new rays to all $x_j^i$ for all $i,j\in I$ and repeat this so that the digraph $D$ constructed has the property that for each $i\in I$ there is a \se\ $g_i$ such that $x_0^ix_1^i\ldots$ is the image of~$R$ under~$g$.
Let $M$ be the set of all \se s of~$D$.
Then the hyperbolic boundary point $\eta$ that contains~$R$ is not the direction of any non-elliptic \se\ by the choice of~$I$, but it lies in $\cL(M)$ since the sequence $(g_i(x_0))_{i\in\N}=(x_i)_{i\in\N}$ converges to it.
Also, no sequence of vertices and hyperbolic boundary points that lie outside of $\{x_i\mid i\in\N\}\cup\{\eta\}$ converges to~$\eta$.
In particular, $\eta$ is not in the closure of $\cD(M)$.
This shows that $\cD(M)$ is not dense in $\cL(M)$ in this situation.
\end{ex}

Despite this negative statement in Example~\ref{ex_DirectionsNotDenseInLimit}, it still feels that the directions of~$M$ come closer to~$\eta$.
We will make this intuition precise in Proposition~\ref{prop_DirectionsAlmostDenseInLimit}.
But before we can prove that, we need some lemmas.

\begin{lem}\label{lem_ellipticFixBalls}
Let $D$ be a rooted hyperbolic digraph with constant out-degree and bounded in-degree.
Let $g$ be an elliptic \se\ of~$D$ and let $n\in\N$.
Then there exists $i\in\N$ such that $g^i$ fixes $\cB^+_n(o)\cup \cB^-_n(o)$ pointwise.
\end{lem}

\begin{proof}
Let $U$ be a finite vertex set that is fixed setwise by~$g$ and let $V$ be the set of all vertices $x$ of~$D$ with $d(x,u)=d(o,U)$ for some $u\in U$.
Note that $V$ is finite, contains~$o$ and is fixed by~$g$ setwise.
Then $g$ induces a permutation on~$V$ since it is a \se.
So there exists $i_0$ such that $g^{i_0}$ fixes~$o$.
Then $g^{i_0}$ induces permutations on $\cB_1^+(o)$ and on $\cB^-_1(o)$.
So there exists $i_1$, a multiple of $i_0$, such that $g^{i_1}$ fixes $\cB_1^+(o)$ and $\cB^-_1(o)$ pointwise.
Recursively, we obtain $i_n$ such that $g^{i_n}$ fixes $\cB_n^+(o)$ and $\cB^-_n(o)$ pointwise.
\end{proof}

As a corollary, we directly obtain the following.

\begin{cor}\label{cor_ellipticIsAuto}
Let $D$ be a rooted hyperbolic digraph with constant out-degree and bounded in-degree.
Every elliptic \se\ of~$D$ is an automorphism.\qed
\end{cor}

Let us now establish a property that guarantees that a \se\ is not elliptic.

\begin{lem}\label{lem_translationIsHyper}
Let $D$ be a rooted hyperbolic digraph with constant out-degree and bounded in-degree.
Let $U$ be a subset of $V(D)$ with $o\notin U$ and let $g$ be a \se\ of~$D$.
If $g(U\cup\{o\})\sub U$, then $g$ is not elliptic and $g^+$ lies in the closure of~$U$.
\end{lem}

\begin{proof}
Let us suppose that $g$ is elliptic.
Then $\{g^i(o)\mid i\in\N\}$ is finite by Lemma~\ref{lem_ellipticFixBalls}.
Since $g$ is a \se, it is injective and hence there must be $i\in\N$ with $g^{i+1}(o)=o$.
But we get recursively that $g^j(o)\in U$ for all $j\in\N$ with  $j\geq 1$.
This contradiction shows that $g$ is not elliptic.

Since all $g^i(o)$ for $i\geq 1$ lie in~$U$, the direction $g^+$ of~$g$ must lie in the closure of~$U$.
\end{proof}

We will now prove that we can `push' hyperbolic boundary points via non-elliptic \se s towards their directions.
This can be seen as an analogue of what is called projectivity for automorphisms of metric spaces, see e.\,g.\ \cite{W-FixedSets}.

\begin{lem}\label{lem_projectivity}
Let $D$ be a rooted hyperbolic digraph with constant out-degree and bounded in-degree.
Let $g$ be a \se\ of~$D$ that is not elliptic.
Let $\eta\in\rand D$ such that $\eta$ is not fixed by any $g^i$ with $i\in\N$.
Then there exist neighbourhoods $U^\circ$ and $V^\circ$ of~$g^+$ and~$\eta$, respectively, such that for all neighbourhoods $U$ and $V$ of $g^+$ and~$\eta$, respectively, with $U\sub U^\circ$ and $V\sub V^\circ$ and finite vertex sets $S_U\sub U$ and $S_V\sub V$, such that $S_U$ meets all $o$-$u$ geodesics for $u\in U$ and $S_V$ meets all $o$-$v$ geodesics for $v\in V$, and there exists $n\in\N$ such that $g^n(V^+)\sub U^+$, where $U^+$ and $V^+$ are the sets of vertices $x$ such that some $o$-$x$ geodesic meets $\cB^+_{\kappa}(S_U)$ and $\cB^+_\kappa(S_V)$, respectively, for
\[
\kappa=(12\delta+4\delta \varphi(\delta+1)+1)\varphi(\delta+1),
\]
where $\varphi\colon\R\to\R$ is a function such that $D$ satisfies (\ref{itm_Bounded1}) and (\ref{itm_Bounded2}) with respect to~$\varphi$.
\end{lem}

\begin{proof}
Let $\cR_g$ be the set of geodesic double rays from hyperbolic boundary points fixed by~$g$ to~$g^+$ and geodesic rays from~$o$ to~$g^+$.
By Lemma~\ref{lem_hyperDi9.4} and because of the definition of hyperbolic digraphs and of the topology of their hyperbolic boundary, there exists an open neighbourhood $U$ of~$g^+$ and a vertex $x_U\in U$ on a geodesic double ray whose subrays are in~$g^+$ such that $S_U:=\cB^+_{6\delta}(x_U)\cap U$ meets every geodesic double ray in~$\cR_g$ whose subrays lie in~$g^+$ and every geodesic (ray) from~$o$ to~$U$ or~$g^+$.
We may assume that $x_U$ lies on one of those geodesic double rays in~$\cR_g$.
Let $V$ be an open neighbourhood of~$\eta$.
We choose $x_V$ and~$S_V$ analogously to $x_U$ and~$S_U$ except that we just take the geodesic rays from~$o$ in~$\eta$ into account instead of~$\cR_g$.
Since there are only finitely many hyperbolic boundary points fixed by~$g$ by Lemmas~\ref{lem_nonEllipticFixFinite1} and~\ref{lem_nonEllipticFixFinite2}, we may assume
\[
(U\cup S_U)\cap (V\cup S_V)=\es
\]
by Lemma~\ref{lem_disjointBNeighbourhoods}.
We may assume that
\begin{equation}\label{itm_projectivity_5}
d(\cR_g,S_V\cup V)>16\delta+\lambda
\end{equation}
for
\[
\lambda:=6\delta+2\delta\varphi(\delta+1)
\]
and, furthermore, we may assume that the geodesics from $S_V$ to $V$ also have at least that distance from~$\cR_g$.
Let $o_V$ be a vertex in $\bigcup \cR_g$ with $d(o_V,S_V\cup V)<\infty$ smallest possible; this vertex exists since $o$ satisfies $d(o,S_V\cup V)<\infty$.
If necessary, we change $U$ and $S_U$ so that $S_U$ also meets all geodesics from~$o_V$ to~$U$, which is possible in the same way we did it for~$o$.
Since $g$ is not elliptic, there exists $n\in\N$ such that $g^n(o)$ and $g^n(o_V)$ lie in~$U$, such that for
\[
\kappa=(12\delta+4\delta \varphi(\delta+1)+1)\varphi(\delta+1)
\]
and for all $v\in S^+_U:=\cB^+_\kappa(S_U)$, we have
\begin{equation}\label{itm_projectivity_3}
(2\delta+(d(o,o_V)+\delta)\varphi(\delta+1))\varphi(\delta+1)<\min\{d(v,g^n(o)),d(v,g^n(o_V))\}
\end{equation}
and
\begin{equation}\label{itm_projectivity_7}
\begin{aligned}
&(\delta+\lambda+d(o,o_V)+(\max\{d(o_V,v)\mid v\in S_V^+\}+\lambda)\varphi(\delta+1))\varphi(\delta+1)\\
<&\min\{d(v,g^n(o),d(v,g^n(o_V))\},
\end{aligned}
\end{equation}
with $S^+_V:=\cB^+_\kappa(S_V)$.
Note that (\ref{itm_projectivity_3}) implies by Proposition~\ref{prop_hyperDi3.3} the following.
\begin{txteq}\label{itm_projectivity_0}
No geodesic that starts either at $g^n(o)$ or at $g^n(o_V)$ and that ends in $g^n(S_V\cup V)$ meets $\cB^+_\delta(S^+_U)$.
\end{txteq}
We may assume that we have chosen $S_U$ and~$U$ as well as $S_V$ and $V$ so that
\begin{equation}\label{itm_projectivity_2}
\min\{d(o,S_V^+),d(o_V,S_V^+)\}>(d(o,o_V)+\lambda)\varphi(\delta+1).
\end{equation}
Let $U^+$ be the set of vertices~$x$ such that some $o$-$x$ geodesic meets~$S^+_U$ and let $V^+$ be the set of vertices~$x$ such that some $o$-$x$ geodesic meets~$S^+_V$.

Let $x\in g^n(V^+)$, let $P_1$ be an $o$-$g^n(o_V)$ geodesic, $P_2$ a $g^n(o_V)$-$x$ geodesic and $P_3$ an $o$-$x$ geodesic.
Let $a$ be a vertex on~$P_1$ in~$S_U$.
Then there exists a vertex $b$ either on~$P_3$ in $\cB^-_\delta(a)$ or on~$P_2$ in $\cB^+_\delta(a)$.

Let us first assume that there is $b\in\cB^-_\delta(a)$ on~$P_3$.
By Lemma~\ref{lem_hyperDi3.4}, there exists $c$ on $P_1\cup P_2$ with $d(c,b)\leq \lambda$.
If $c$ lies on either $aP_1$ or on $P_2$, then there exists a directed $a$-$b$ path and hence Proposition~\ref{prop_hyperDi3.3} implies $d(a,b)\leq (\delta+1)\varphi(\delta+1)\leq\kappa$.
So $b$ lies in~$S_U^+$ and hence $x\in U^+$.
If $c$ lies on $P_1a$, then we have $d(c,a)\leq\delta+\lambda$.
Let $d$ be on~$P_3$ with $d(b,d)=\delta+\lambda$, if possible, and $d=x$ otherwise.
Again, by Lemma~\ref{lem_hyperDi3.4}, there exists a vertex $c_d$ on $P_1\cup P_2$ with $d(c_d,d)\leq\lambda$.
This vertex must lie on either $aP_1$ or on $P_2$, since $d(o,b)\geq d(o,a)-\delta$.
So there exists a directed $a$-$d$ path and we apply Proposition~\ref{prop_hyperDi3.3} in order to obtain
\begin{align*}
d(a,d)&\leq (d(b,a)+d(b,d))\varphi(\delta+1)\\
&\leq (2\delta +\lambda)\varphi(\delta+1)
\end{align*}
which shows $d\in S_U^+$ and hence $x\in U^+$.

Let us now assume that there exists $b\in\cB^+_\delta(a)$ on~$P_2$.
We consider a geodesic triangle with end vertices $g^n(o)$, $g^n(o_V)$ and~$x$ such that $P_2$ is the side from $g^n(o_V)$ to~$x$.
We may assume that there exists $y\in g^n(S_V^+)$ on the side from $g^n(o)$ to~$x$.
We have $d(z,y)\leq\lambda$ for some $z$ on~$P_2$, since Proposition~\ref{prop_hyperDi3.3} leads to a contradiction to (\ref{itm_projectivity_2}) otherwise.

If $b$ lies on $g^n(o_V)P_2z$ or if $d(a,z)\leq \delta+\lambda$, then we have with $b'=b$ in the first case and $b'=z$ in the second case the following
\begin{align*}
d(a,g^n(o))&\leq (d(a,b')+d(g^n(o),b'))\varphi(\delta+1)\\
&\leq (\delta+\lambda+d(g^n(o),g^n(o_V))+d(g^n(o_V),b'))\varphi(\delta+1)\\
&\leq (\delta+\lambda+d(o,o_V)+d(g^n(o_V),z))\varphi(\delta+1)\\
&\leq (\delta+\lambda+d(o,o_V)+(d(g^n(o_V),y)+\lambda)\varphi(\delta+1))\varphi(\delta+1)\\
&\leq (\delta+\lambda+d(o,o_V)+(\max\{d(o_V,v)\mid v\in S_V^+\}+\lambda)\varphi(\delta+1))\varphi(\delta+1),
\end{align*}
which contradicts~(\ref{itm_projectivity_7}).

Now let $b$ be on $zP_2$ and assume $d(a,z)>\delta+\lambda$.
Considering the geodesic triangle with $z$, $y$ and~$x$ as end vertices, we conclude that $b$ must lie within the $\delta$-in-ball around the side from~$y$ to~$x$.
Thus, $a$ is a vertex on~$\cR_g$ with distance at most $2\delta$ to a vertex on a $g^n(S_V^+)$-$g^n(V^+)$ geodesic.
This is a contradiction to the choice of that distance as chosen after (\ref{itm_projectivity_5}).
This shows that the case that some $b\in\cB^+_\delta(a)$ lies on~$P_2$ always leads to a contradiction.
Thus, we have shown the assertion, that is, we have shown $g^n(V^+)\sub U^+$.

We note that the choices of~$V$ and~$U$ only depended on large enough distances from certain fixed vertex sets.
Thus, for all smaller neighbourhoods of~$g^+$ and~$\eta$, the statement stays true, which shows the assertion.
\end{proof}

As a corollary of Lemma~\ref{lem_projectivity}, we obtain the following.

\begin{cor}\label{cor_ellipticImageOfNonElliptic}
Let $D$ be a rooted hyperbolic digraph with constant out-degree and bounded in-degree.
Let $g$ and $h$ be \se s of~$D$ such that $g$ is elliptic and $h$ is not elliptic and such that $g(h^+)$ is not fixed by~$h$.
Then there exists $i\in\N$ such that $gh^i$ is not elliptic.

Furthermore, if $U$ is a neighbourhood of~$h^+$ and $S_U$ a finite vertex set that meets every $o$-$U$ geodesic, then $(gh^i)^+$ lies in the closure of $g(U^+)$, where $U^+$ is the set of vertices $x$ such that some $o$-$x$ geodesic meets $\cB^+_{\kappa}(S_U)$ for
\[
\kappa=(12\delta+4\delta \varphi(\delta+1)+1)\varphi(\delta+1),
\]
where $\varphi\colon\R\to\R$ is a function such that $D$ satisfies (\ref{itm_Bounded1}) and (\ref{itm_Bounded2}) with respect to~$\varphi$.
\end{cor}

\begin{proof}
Let $U^\circ$ and $V^\circ$ be obtained by Lemma~\ref{lem_projectivity}.
Then there exists an open neighbourhood $U$ of~$h^+$ such that $U\sub U^\circ$ and $g(U)\sub V^\circ$.
Let $n\in\N$ be obtained by Lemma~\ref{lem_projectivity} for~$U$.
Then we have $gh^n(g(U^+))\sub g(U^+)$.
Thus, Lemma~\ref{lem_translationIsHyper} implies that the \se\ $gh^n$ is not elliptic and $(gh^n)^+$ lies in the closure of~$g(U^+)$.
\end{proof}

We have already mentioned that, despite of Example~\ref{ex_DirectionsNotDenseInLimit}, the directions are spread out quite well among the limit set.
We will now prove that in the following proposition.

\begin{prop}\label{prop_DirectionsAlmostDenseInLimit}
Let $D$ be a rooted hyperbolic digraph with constant out-degree and bounded in-degree and let $M$ be a monoid of \se s of~$D$.
Let $|\cL(M)|\geq 2$ and let $\eta\in\cL(M)$.
Then every for open neighbourhood $U$ of~$\eta$ there is an element $\mu\in\cD(M)$ such that there exists a geodesic ray in~$\mu$ that starts at~$o$ has a vertex in~$U$.
\end{prop}

\begin{proof}
Let $\varphi\colon\R\to\R$ be a function such that $D$ satisfies (\ref{itm_Bounded1}) and (\ref{itm_Bounded2}) with respect to~$\varphi$.
Let $(g_i)_{i\in\N}$ be a sequence in~$M$ such that $(g_i(o))_{i\in\N}$ converges to~$\eta$.
Let us suppose that there is an open neighbourhood $U$ of~$\eta$ such that for all $\mu\in\cD(M)$ all geodesic rays in~$\mu$ that start at~$o$ are disjoint from~$U$.
By considering a subsequence of $(g_i)_{i\in\N}$, we may assume that either all $g_i$ are elliptic or none of them is.

\medskip

Let us first assume that no $g_i$ is elliptic.
Let $R$ be a geodesic ray defined by $o$-$g_i(o)$ geodesics.
In particular, $R$ starts at~$o$.
Let $x$ be on~$R$ such hat $\cB_\delta^-(x)\sub U$.
This vertex exists since $d_h$ is a visual pseudo-semimetric.
Let $i\in\N$ such that $x$ lies on an $o$-$g_i(o)$ geodesic that was used for the definition of~$R$ and such that
\begin{equation}\label{itm_DirectionsAlmostDenseInLimit_1}
(d(o,x)+2\delta)\varphi(\delta+1)<d(x,g_i^k(o))
\end{equation}
for all $k\geq 1$.
We consider a triangle with end vertices $o$, $g_i(o)$ and some vertex $z$ on a geodesic ray $Q$ in~$g_i^+$ that starts at~$o$ and avoids~$U$ such that $d(g_i(o),z)<\infty$.
Note that eventually all vertices on that ray satisfy the additional requirement.
Then there exists $y$ on the side from $g_i(o)$ to~$z$ with $d(x,y)\leq\delta$ since $Q$ avoids~$U$ and $\cB^-_\delta(x)\sub U$.
Then (\ref{itm_DirectionsAlmostDenseInLimit_1}) and Proposition~\ref{prop_hyperDi3.3} imply
\[
d(o,x)+\delta<d(x,g_i(o))/\varphi(\delta+1)-\delta\leq d(g_i(o),y).
\]
Thus, we have
\[
d(o,z)\leq d(o,x)+\delta+d(y,z)<d(g_i(o),y)+d(y,z)=d(g_i(o),z).
\]
Applying this argument recursively, we obtain
\[
d(g_i^k(o),z)<d(g_i^{k+1}(o),z)
\]
for all those $k\in\N$ with $d(g_i^{k+1}(o),z)<\infty$.
Since we may have chosen $z$ within any neighbourhood of~$g_i^+$, this shows that the sequence $(g_i^k(o))_{k\in\N}$ does not converge to~$g_i^+$.
This contradiction finishes the case that no $g_i$ is elliptic.

\medskip

Let us now assume that all $g_i$ are elliptic.
As in the proof of Lemma~\ref{lem_projectivity}, we may assume that there exists a finite subset $S_U$ of~$U$ such that all geodesic rays in~$\eta$ that start at~$o$ contain a vertex of~$S_U$.
Set $S_U^+:=\cB^+_\kappa(S_U)$ for
\[
\kappa:=(2\lambda+1)\varphi(\delta+1)+\lambda
\]
with $\lambda:=6\delta+2\delta\varphi(\delta+1)$.
Let $U^+$ be the set of all vertices~$u$ such that some $o$-$u$ geodesic  meets~$S_U^+$.
Let $\mu\in\cL(M)$ be distinct from~$\eta$ and let $V$ be an open neighbourhood of~$\mu$.
Define $S_V$, $S_V^+$ and $V^+$ analogously to $S_U$, $S_U^+$ and $U^+$, respectively.
Let $(h_i)_{i\in\N}$ be a sequence in~$M$ such that $(h_i(o))_{i\in\N}$ converges to~$\mu$.
By taking suitable subsequences, we may assume that $S_U$ contains vertices of some $o$-$g_i(o)$ geodesic for every $i\geq 1$ and that $S_V$ contains vertices of some $o$-$h_i(o)$ geodesic for every $i\geq 1$.

Let $i\geq 1$ and $x\in g_i(S_U^+\cup S_V^+)$.
Let $Q_1$ be an $o$-$g_i(o)$ geodesic that contains a vertex of $S_U$, $Q_2$ a $g_i(o)$-$x$ geodesic and $Q_3$ an $o$-$x$ geodesic.
Then all but at most the last $d(g_i(o),x)+\delta$ vertices of~$Q_3$ lie in $\cB^+_\delta(Q_1)$ as $Q_3$ is geodesic.
Thus, for $i$ large enough, this implies that there are $u,v$ on~$Q_3$ such that $v$ is the out-neighbour of~$u$ on~$Q_3$ and $u$ lies within the $\delta$-out-ball of the directed subpath of~$Q_1$ from $o$ to the first vertex $y$ of~$S_U$ and $v$ within the $\delta$-out-ball around $yQ_1g_i(o)$ but not in $\cB^+_\delta(y)$.
Applying Proposition~\ref{prop_hyperDi3.3}, we obtain that $Q_3$ contains a vertex of distance at most $(2\delta+1)\varphi(\delta+1)\leq\kappa$ from~$S_U$.
So $P_3$ contains a vertex of~$S_U^+$.
Thus, we have shown the following.
\begin{txteq}\label{itm_DirectionsAlmostDenseInLimit_2}
If $i$ is large enough, then we have $g_i(S_U^+\cup S_V^+)\sub U^+$.
\end{txteq}
Analogously, we obtain the following.
\begin{txteq}\label{itm_DirectionsAlmostDenseInLimit_3}
If $i$ is large enough, then we have $h_i(S_U^+\cup S_V^+)\sub V^+$.
\end{txteq}
So let us now assume that $i$ is large enough such that is satisfies $g_i(S_U^+\cup S_V^+)\sub U^+$ and $h_i(S_U^+\cup S_V^+)\sub V^+$.

Since $g_i$ is elliptic, Lemma~\ref{lem_translationIsHyper} implies that there exists $y\in g_i(U^+)\sm U^+$.
Let $P_1=Q_1$, let $P_2$ be a $g_i(o)$-$y$ geodesic and let $P_3$ be an $o$-$y$ geodesic.
Let $a$ be the last vertex on $P_3$ such that there exists $a'$ on~$P_1$ before $S_U$ with $d(a',a)\leq \lambda$ and let $a^+$ be the out-neighbour of~$a$ on~$P_3$.
By Lemma~\ref{lem_hyperDi3.4}, the vertex $a^+$ lies in the out-ball of radius $\lambda$ around $P_1$ or~$P_2$.
If it lies in that out-ball around~$P_1$, then a vertex $a''$ on~$P_1$ with $d(a'',a^+)\leq \lambda$ must lie in or after $S_U$.
Thus, we can apply Proposition~\ref{prop_hyperDi3.3} and obtain
\[
d(a',a'')\leq (2\lambda+1)\varphi(\delta+1)\leq\kappa-\lambda,
\]
which implies $a^+\in \cB^+_\kappa(S_U)$, a contradiction to the choice of~$y$.
Thus, there exists $b_y$ on~$P_2$ with $d(b_y,a^+)\leq \lambda$.

We have
\[
d(o,a)\leq \max\{d(o,v)\mid v\in S_U\}+\lambda.
\]
Thus, and since any $o$-$b_y$ geodesic $P_4$ lies in $\cB^+_\delta(o,a^+)\cup \cB^-_\delta(b_y,a^+)$, we have
\begin{align*}
&d(o,b_y)\\
\leq\ &d(o,a)+1+\delta+(\delta+\lambda)\varphi(\delta+1)\\
\leq\ &\max\{d(o,v)\mid v\in S_U\}+\delta+\lambda+1+(\delta+\lambda)\varphi(\delta+1).
\end{align*}
Since $g_i(o)P_2 b_y$ lies in $\cB^+_\delta(P_1)\cup \cB^-_\delta(P_4)$ and since the distance from~$o$ to $\cB^-_\delta(P_4)$ is at most $(d(o,b_y)+\delta)\varphi(\delta+1)$ by Proposition~\ref{prop_hyperDi3.3}, all vertices of $g_i(o)P_2 b_y$ that are further away from~$o$ must lies in $\cB^+_\delta(P_1)$.

\medskip

If $g_i(V^+)\sub U^+$ and $h_i(U^+)\sub V^+$, then $g_ih_i(U^+)\sub U^+$ and $g_ih_i(o)\in U^+$.
So Lemma~\ref{lem_translationIsHyper} implies that $g_ih_i$ is not elliptic.
Since all $g_ih_i(o)$ lie in~$U^+$, there exists a geodesic ray in $(g_ih_i)^+$, one that is defined by infinitely many directed $o$-$(g_ih_i)^n(o)$ paths, that lies eventually in $U^+$.
Since $d_h$ is visual, we may have chosen $U$ and~$S_U$ such that $U^+$ lies in a neighbourhood $U^\circ$ of~$\eta$ which avoids all geodesic rays from~$o$ to elements of $\cD(M)$, similar the way we asked it for~$U$.
But the geodesic ray from~$o$ that lies in $(g_ih_i)^+$ is not disjoint from~$U^\circ$, a contradiction that shows that we have either $g_i(V^+)\not\sub U^+$ or $h_i(U^+)\not\sub V^+$.

\medskip

First, let us assume $g_i(V^+)\not\sub U^+$.
Then there exists $z\in g_i(V^+)\sm U^+$.
By an analogous argument as for~$y$, we find $b_z$ on a $g_i(o)$-$z$ geodesic $P_5$ with similar properties as~$b_y$, in particular, that all vertices on~$P_5$ that have distance more than $(d(o,b_y)+\delta)\varphi(\delta+1)$ from~$o$ must lie in $\cB^+_\delta(P_1)$.

Let $R_\eta$ and $R_\mu$ be geodesic rays in~$\eta$ and~$\mu$, respectively, that start at~$o$.
These exist by \cite[Proposition 9.3]{H-HyperDiBound}.
For every $n\in\N$, there exists $i_n\in\N$ such that the subpaths $P_\eta^n$ and $P_\mu^n$ of~$R_\eta$ and~$R_\mu$, respectively, of lengths~$n$ starting at~$o$ satisfy the following: $g_{i_n}(P^n_\eta)$ and $g_{i_n}(P^n_\mu)$ lie in $\cB^+_\delta(P^n)$, where $P^n$ is an $o$-$g_{i_n}(o)$ geodesic.
Since $g_i$ is elliptic, Lemma~\ref{lem_ellipticFixBalls} implies the existence of some $m\in\N$ with $g_{i_n}^m(P_\eta^n)=P_\eta^n$ and $g_{i_n}^m(P_\mu^n)=P_\mu^n$.
The paths $g_{i_n}^m(P^n)$ define a geodesic anti-ray~$Q$ that ends at~$o$.
Let $\omega$ be the hyperbolic boundary point that contains~$Q$.
Note that there are infinitely many directed $Q$-$R_\eta$ paths and $Q$-$R_\mu$ paths of length at most~$\delta$.
This implies $\eta=\omega=\mu$ by Corollary~\ref{cor_corOf14.1}, in contradiction to our assumption $\eta\neq\mu$.

\medskip

So let us now consider the case that $h_i(U^+)\not\sub V^+$.
If there are infinitely many $h_i$ that are elliptic, then we may assume that all $h_i$ are elliptic and an argument that is completely analogous to the previous one leads to a contradiction.
Thus, we may additionally assume that no $h_i$ is elliptic.
We may assume in this situation that $h_j=h_1^j$.
So in particular, we have $h_1^+=h_j^+$ for all $i,j\in\N$.

If $g_i$ does not fix $h_1^+$, then there is an open neighbourhood $W\sub V$ of~$h_1^+$ such that $W$ and $g_i(W)$ are disjoint.
Corollary~\ref{cor_ellipticImageOfNonElliptic} implies that there exists $n\in\N$ such that $f_i:=g_ih_i^n$ is not elliptic and $f_i^+$ lies in the closure of $g_i(W)$.
Since $g_i(V^+)\sub U^+$, our assertion holds in this case as we have shown above, since we may assume that $U^+$ lies in a neighbourhood $U^\circ$ of~$\eta$ that avoids all rays from~$o$ in elements of $\cD(M)$.

So let us assume that all $g_i$ fix $h_1^+$.
Let us first assume that no $h_i$ fixes~$\eta$.
We may have chosen $U$ and~$V$ so that they are obtained by Lemma~\ref{lem_projectivity}.
Then for $n\in\N$ from that lemma, we obtain $h_i^n(U^+)\sub V^+$.
This is a contradiction to our assumption since $h_i^n=h_{ni}$.

Let us now assume that some $h_i$ fixes~$\eta$.
There exists a geodesic ray $R_h$ in $h_1^+$ starting at~$o$ by \cite[Proposition 9.3]{H-HyperDiBound}.
We consider all images $g_i^n(R_h)$.
Since they all lie in $h_1^+$, infinitely many of them share a common vertex among which infinitely many share a common in- and a common out-neighbour.
Recursively, we obtain a geodesic double ray~$Q$ whose subrays lie in~$h_1^+$.
Considering a geodesic triangle with end vertices $o$, $g_i(o)$ and $h_i(o)$, the anti-rays in~$Q$ must lie in some $\omega\in\rand D$ with $d(\eta,\omega)=0$.
Then Corollary~\ref{cor_nonEllipticFixFinite2} implies that $h_i$ is hyperbolic, in particular, that $\eta$ and $h_i^+$ are the only hyperbolic boundary points fixed by~$h_i$.
By Lemma~\ref{lem_PowerHyperPara}, all $h_j$ fix $\eta$ and~$h_1^+$.
Let $\cR$ be the set of all geodesic double rays from~$\eta$ to~$h_1^+$.
Since all $h_j$ and all $g_j$ fix $\eta$ and $h_1^+$, they must also leave $\cR$ invariant.
Let
\[
M:=\max\{d(o,R)\mid R\in \cR\}
\]
and let $u$ be a vertex on a double ray $R\in \cR$ with $d(o,u)=d(o,R)$.
We may assume that $d(o,S_U)>M+\delta$.
Let $v$ be a vertex on~$R$ that lies in~$U$; note that this vertex exists since some anti-subray of~$R$ lies in~$\eta$.
Considering the geodesic triangle with end vertices $o$, $u$ and~$v$ such that $R_1$ is an $o$-$u$ geodesic, $R_2$ is an $o$-$v$ geodesic and $R_3=vRu$.
Then there must be a vertex on~$R_3$ that lies in the $\delta$-out-ball of a vertex in $V(R_2)\cap S_U$.
This implies that $R$ meets $S_U^+$.
Since $(g_j(o))_{j\in\N}$ converges to~$\eta$, we find $j\in\N$ such that every vertex on the elements of~$\cR$ that lie in $\cB^+_M(g_j(o))$ lie in~$U$.
Thus, the elements of~$\cR$ first meet $g_j(S_U^+)$, $\cB^+_M(g_j(o))$, $S_U^+$, and $\cB^+_M(o)$ in this order.
This implies that the same is true for all powers of~$g_j$.
In particular, we do not find any $n\in\N$ with $g_j^n(o)=o$.
This contradiction to Lemma~\ref{lem_ellipticFixBalls} finishes the last case of this proof.
\end{proof}

Now we are able to determine all possible values of the numbers of limit points and the numbers of directions of monoids of \se s.

\begin{thm}\label{thm_numberOfLimits}\label{thm_numberOfDirections}
Let $D$ be a rooted hyperbolic digraph with constant out-degree and bounded in-degree and let $M$ be a monoid of \se s of~$D$.
The sets $\cL(M)$ and $\cD(M)$ have either none, one, two, or infinitely many elements.
\end{thm}

\begin{proof}
Let us assume that $|\cL(M)|>2$.
Then we also have $|\cD(M)|>2$ by Proposition~\ref{prop_DirectionsAlmostDenseInLimit}.
Let us show the following.
\begin{txteq}\label{itm_numberOfLimits_1}
There exist non-elliptic \se s $f,g$ with $g^n(f^+)\neq f^+$ for all $n\in\N$ with $n\geq 1$.
\end{txteq}
Let $f,g,h\in M$ be non-elliptic such that all of their directions are distinct.
If some $f^n$ with $n\in\N$ and $n\geq 1$ fixes $g^+$, then there exists $\eta\in\rand D$ with $d(f^+,\eta)=0$ that contains anti-rays by Lemma~\ref{lem_2BPfixed-a}.
If additionally some $g^m$ with $m\in\N$ and $m\geq 1$ fixes $f^+$, then $g$ or $g^m$ is hyperbolic by Corollary~\ref{cor_nonEllipticFixFinite2} and so are all $g^k$ for $k\in\N$ with $k\geq 1$ by Lemma~\ref{lem_PowerHyperPara}.
So none of these $g^k$ fix~$h^+$.
This finishes the proof of~(\ref{itm_numberOfLimits_1}).

Let us prove that $g^i(f^+)\neq g^j(f^+)$ for all distinct $i,j\in\N$.
Suppose that there are $i<j$ with $g^i(f^+)=g^j(f^+)$.
Set $m:=j-i$.
Then $g^m$ fixes $g^i(f^+)$.
If $i>0$, then Proposition~\ref{prop_SeInduceInjection} implies that $g^m$ must fix $g^{i-1}(f^+)$, too.
Inductively, we obtain $g^m(f^+)=f^+$, a contradiction to~(\ref{itm_numberOfLimits_1}).
Thus, the set $\{g^i(f^+)\mid i\in\N\}$ is infinite.

Since the sequence $(g^if^j(o))_{j\in\N}$ converges to $g^i(f^+)$ for every $i\in\N$, we obtain that $\cL(M)$ is infinite.
So $\cD(M)$ is infinite, too, by Proposition~\ref{prop_DirectionsAlmostDenseInLimit}.
\end{proof}

Let us now prove our fixed point theorem: either the monoid of \se s fixes a finite vertex set or a unique limit point or it has exactly two limit points or the non-elliptic \se s satisfy a certain condition which is enough so prove that the monoid contains a free submonoid as we will see later (Theorem~\ref{thm_freeSubmonoid}).

\begin{thm}\label{thm_fpa}
Let $D$ be a rooted hyperbolic digraph with constant out-degree and bounded in-degree and let $M$ be a monoid of \se s of~$D$.
Then one of the following holds.
\begin{enumerate}[\rm (i)]
\item\label{itm_fpa_1} $M$ fixes a finite set of vertices.
\item\label{itm_fpa_2} $M$ fixes a unique element of $\cL(M)$.
\item\label{itm_fpa_3} $\cL(M)$ consists of exactly two elements.
\item\label{itm_fpa_4} $M$ contains two elements that are not elliptic and that do not fix the direction of the other.
\end{enumerate}
\end{thm}

\begin{proof}
Obviously, $|\cL(M)|=1$ implies (\ref{itm_fpa_2}) and $|\cL(M)|=2$ implies (\ref{itm_fpa_3}).
Let us assume that $|\cL(M)|=0$ holds.
Then every element of~$M$ must be elliptic.
First note that every element of~$M$ is an automorphism by Corollary~\ref{cor_ellipticIsAuto}, so we have $d(g(o),v)<\infty$ for all $v\in V(D)$ as $d(o,v)<\infty$ for all $v\in V(D)$.
If $\{g(o)\mid g\in M\}$ is not finite, then sequential compactness of $D\cup \rand D$ implies that an infinite sequence $(g_i(o))_{i\in\N}$ in $\{g(o)\mid g\in M\}$ converges to some $\eta\in\rand D$.
Let $R$ be the ray starting at~$o$ that is defined by $o$-$g_i(o)$ geodesics.
Then $R$ lies in a hyperbolic boundary point $\mu$ with $d_h(\mu,\eta)=0$.
Since we find $g_i(o)-R$ geodesics of arbitrary distance from~$o$, we also have $d_h(\eta,\mu)=0$ as $d_h$ is visual.
Thus, we have $\eta\in\cL(M)$ contrary to our assumption that $\cL(M)$ is empty.

Let us now assume that $|\cL(M)|>2$ holds and hence $|\cD(M)|>2$ by Proposition~\ref{prop_DirectionsAlmostDenseInLimit}.
Then both cardinalities are infinite by Theorem~\ref{thm_numberOfLimits}.
If $M$ fixes some finite set of vertices, then $\cL(M)$ must be empty.
Thus, $M$ does not fix any finite set of vertices.
Let us suppose that (\ref{itm_fpa_4}) does not hold, i.\,e.\ that there are no two non-elliptic elements each of which fixes the direction of the other.

Let us show the following.
\begin{txteq}\label{eq_fpa_1}
If $g_1,g_2,g_3$ are non-elliptic elements of~$M$ with pairwise distinct directions, then there are no two of them that fix the direction of the third one.
\end{txteq}
Let us suppose that $g_1$ and $g_3$ fix the direction of~$g_2$.
Then there are $g_1^-, g_3^-\in\rand D$ with $d(g_1^+,g_1^-)=0$ and $d(g_3^+,g_3^-)=0$ such that $g_1^-$ and $g_3^-$ contain anti-rays by Lemma~\ref{lem_2BPfixed-a}.
Since either $g_1$ fixes $g_3^+$ or $g_3$ fixes $g_1^+$, Corollary~\ref{cor_nonEllipticFixFinite2} implies that either $g_1$ or $g_3$ is hyperbolic and thus cannot fix~$g_2^+$, which is a contradiction.
This shows~(\ref{eq_fpa_1}).

Let $g_1,g_2,g_3,g_4\in M$ be non-elliptic with pairwise distinct directions.
We may assume that $g_1$ fixes~$g_2^+$.
Then (\ref{eq_fpa_1}) implies that neither $g_3$ nor~$g_4$ fixes~$g_2^+$, so $g_2$ fixes~$g_3^+$ and~$g_4^+$.
We have that either $g_3$ fixes $g_4^+$ or $g_4$ fixes~$g_3^+$.
But in both cases, we obtain a contradiction to~(\ref{eq_fpa_1}) since in the first case $g_2$ and $g_3$ fix $g_4^+$ while in the second case $g_2$ and $g_4$ fix~$g_3^+$.
This contradiction shows that (\ref{itm_fpa_4}) holds in our situation.
\end{proof}

Contrary to the situation of automorphism groups on (hyperbolic) graphs, the statements of our fixed point theorem are not mutually exclusive.
The reason for that is that we only talk about the directions of non-elliptic elements and not about all of their fixed points.
For automorphisms, the last statement in Theorem~\ref{thm_fpa} can be stated as follows: `There are two hyperbolic automorphisms that freely generate a free subgroup.'
While we also obtain a statement about free submonoid (Theorem~\ref{thm_freeSubmonoid}), the mutual exclusiveness is out of reach as the following example shows.

\begin{ex}
Let $D$ be digraph that is obtained from a tree by placing each edge by two inversely oriented edges.
This is a hyperbolic digraph.
Let $\eta$ be any hyperbolic boundary point and $M$ be the monoid of \se s that fix~$\eta$.
Obviously, $M$ satisfies (\ref{itm_fpa_2}) from Theorem~\ref{thm_fpa}.
Let $\mu,\nu\in\rand D\sm\{\eta\}$ be distinct.
Then there are non-elliptic \se s with $\mu$ and $\nu$ as directions.
Obviously, they have $\eta$ and $\mu$ or $\eta$ and $\nu$ as the only hyperbolic boundary points fixed by them, which shows that this example also satisfies (\ref{itm_fpa_4}) from Theorem~\ref{thm_fpa}.
\end{ex}

Before we will prove that (\ref{itm_fpa_4}) of Theorem~\ref{thm_fpa} implies that the monoid contains a free submonoid, we need a variant of the ping-pong-lemma for semigroups that guarantees us to find a free monoid.

\begin{lem}\label{lem_PingPong}
Let $M$ be a cancellative semigroup of \se s of a digraph~$D$.
Let $m_1, m_2 \in M$ and subsets $U, V \sub V(D)$ such that
\[
U \cap V = \es,\qquad m_1(U \cup V) \sub U, \qquad m_2(U \cup V) \sub V.
\]
Then, $m_1$ and $m_2$ freely generate a free semigroup.
\end{lem}

\begin{proof}
Let $w$ and $w'$ be words over $\{m_1,m_2\}$ such that their evaluated \se s of~$D$ coincide.
If $w$ is empty, then this can only happen by the assumptions if $w'$ is empty, too.
Let $a,a'$ be the first letters of $w, w'$, respectively, and let $v,v'$ be words such that $av$ is the word $w$ and $a'v'$ is the word~$w'$.
By the assumptions, we must have $a=a'$.
Since $M$ is cancellative, this implies that the elements of~$M$ corresponding to the words $v$ and~$v'$ coincide in their evaluated \se s of~$D$.
By induction on the length of~$w$, the assumption follows.
\end{proof}

\begin{thm}\label{thm_freeSubmonoid}
Let $D$ be a rooted hyperbolic digraph with constant out-degree and bounded in-degree and let $M$ be a monoid of \se s of~$D$.
If $|\cL(M)|>2$ and $g,h\in M$ are non-elliptic \se s that do not fix the direction of the other, then there are $n,m\in \N$ such that $g^n$ and $h^m$ freely generate a free submonoid of~$M$.
\end{thm}

\begin{proof}
Let $\varphi\colon\R\to\R$ be a function such that $D$ satisfies (\ref{itm_Bounded1}) and (\ref{itm_Bounded2}) with respect to~$\varphi$.
By Lemmas~\ref{lem_nonEllipticFixFinite1} and~\ref{lem_nonEllipticFixFinite2}, there are only finitely many hyperbolic boundary points fixed by~$g$ and by~$h$.
So by Lemma~\ref{lem_disjointBNeighbourhoods}, there are disjoint open neighbourhoods $U$ and~$V$ of $g^+$ and~$h^+$, respectively, such that $U$ contains no hyperbolic boundary point fixed by~$h$ and $V$ contains no hyperbolic boundary point fixed by~$g$.

Applying Lemma~\ref{lem_projectivity} twice, once for $g$ and $g^+$ and once for $h$ and~$h^+$, we may assume that there are $m,n\in\N$ such that with the notations $S_U$ and $U^+$ as well as $S_V$ and $V^+$ from Lemma~\ref{lem_projectivity} we have $g^m(V^+)\sub U^+$ and $h^n(U^+)\sub V^+$.
We set $S_U^+:=\cB^+_\kappa(S_U)$ and $S_V^+:=\cB^+_\kappa(S_V)$ with
\[
\kappa=(12\delta+4\delta \varphi(\delta+1)+1)\varphi(\delta+1).
\]
Note that all but finitely many $g^i(o)$ lie in~$U$ and all but finitely many $h^i(o)$ lie in~$V$.
Let us show the following.
\begin{txteq}\label{eq_freeSubmonoid_1}
We may assume that $U^+$ and $V^+$ are disjoint.
\end{txteq}

If there exists $x\in U^+\cap V^+$, then let $P_1$ be an $o$-$x$ geodesic that meets $S_U^+$ and let $P_2$ be an $o$-$x$ geodesic that meets $S_V^+$.
Let $c\in S_U^+$ be on $P_1$ and let $a\in S_U$ with $d(a,c)\leq\kappa$.
Similarly, let $d\in S_V^+$ be on $P_2$ and let $b\in S_V$ with $d(b,d)\leq\kappa$.
Let $P_3$ be an $o$-$a$ geodesic, $P_4$ an $a$-$x$ geodesic, $P_5$ an $o$-$b$ geodesic and $P_6$ a $b$-$x$ geodesic.
Let $y_1$ be on~$P_1$ with
\begin{equation}\label{eq_freeSubmonoid_2}
d(o,y_1)<\max\{d(o,c),d(o,d)\}-(\kappa+2\delta+\lambda)
\end{equation}
with $\lambda:=6\delta+2\delta \varphi(\delta+1)$.
Since $y_1\in B^+_\delta(P_3)\cup B^-_\delta(P_4)$, there must be $y$ on~$P_3$ with $d(y,y_1)\leq\delta$, since $P_4$ is bounded by $d(c,x)+\kappa$ and we would obtain a contradiction to $P_1$ being geodesic otherwise.
Considering a geodesic triangle with $o$, $x$ and~$o$ as end vertices, there exists $y_2$ on~$P_2$ with $d(y_1,y_2)\leq\delta$.
Applying Lemma~\ref{lem_hyperDi3.4}, there exists $z$ on $P_5\cup P_6$ with $d(y_2,z)\leq \lambda$.
If $z$ lies on~$P_6$, then we have
\[
d(y_2,x)\leq\lambda+d(b,x)\leq\kappa+\lambda+d(d,x),
\]
which is a contradiction to~(\ref{eq_freeSubmonoid_2}) and to the fact that $P_2$ is geodesic.
Thus, $z$ lies on~$P_5$.
Now Proposition~\ref{prop_hyperDi3.3} implies that $y$ and~$z$ are far away from~$o$ in terms of $d(o,y_1)$: their distance from~$o$ is at least $d(o,y_1)/\varphi(\delta+1)-\delta-\lambda$ because of
\[
d(o,y_1)\leq(d(o,y)+d(y,y_1))\varphi(\delta+1)
\]
in the situation for~$y$ and similarly for~$z$.
Since we may have chosen $U$ and~$V$ of arbitrarily large distance from~$o$ according to Lemma~\ref{lem_projectivity}, this implies that geodesic rays from~$o$ that lie in~$g^+$ and in~$h^+$ must be equivalent.
Thus, we have $g^+=h^+$, a contradiction that shows that there exists $U$ and~$V$ such that $U^+$ and~$V^+$ are disjoint, which finishes the proof of~(\ref{eq_freeSubmonoid_1}).

In order to obtain the assertion, we simply apply Lemma~\ref{lem_PingPong}.
\end{proof}

\section{Hyperbolic monoids}\label{sec_monoids}

Let $M$ be a finitely generated semigroup.
We are considering the \emph{right} Cayley digraphs, i.\,e.\ for generators $s$ and $m\in M$, we have directed edges from $m$ to~$ms$.
Note that every element of~$M$ induces a \se\ of that Cayley digraph by left-multiplication.
If $M$ is right cancellative, then each of its Cayley digraphs with respect to finite generating sets satisfies (\ref{itm_Bounded1}) and~(\ref{itm_Bounded2}).

Our first result will be that infinite right cancellative finitely generated hyperbolic\linebreak monoids contain elements of infinite order.
We will see later (Theorem~\ref{thm_infOrderQI}) that the embedding of the submonoid generated by such an element of infinite order defines a \qi\ embedding into the monoid.

\begin{thm}\label{thm_InfOrderElement}
Let $M$ be an infinite cancellative finitely generated hyperbolic\linebreak monoid.
Then it contains an element of infinite order.
\end{thm}

\begin{proof}
Let $D$ be a locally finite hyperbolic Cayley digraph of~$M$ and let $o$ be the vertex corresponding to the identity of~$M$.
By Proposition~\ref{prop_DirectionsAlmostDenseInLimit}, the statement holds if $|\cL(M)|\geq 2$.
So let us assume that $|\cL(M)|\leq 1$.
Since $D$ is infinite, there exists a hyperbolic boundary point and, since $d(o,x)<\infty$ for all $x\in V(D)$, there exists a hyperbolic boundary point that contains rays.
This lies in $\cL(M)$.
So there is a unique element in $\cL(M)$.
Since all hyperbolic boundary points that contains rays lie in~$\cL(M)$, this implies that $D$ has a unique hyperbolic boundary point~$\eta$ that contain rays.
Let $R=x_0x_1\ldots$ be a geodesic ray in~$\eta$ that starts at~$o$.
Note that this ray exists by \cite[Proposition 9.3]{H-HyperDiBound}.

Let us suppose that $M$ contains no element of infinite order.
Then every element is elliptic.
So Corollary~\ref{cor_ellipticIsAuto} implies that every element of~$M$ is an automorphism on~$D$.
Let $m_i\in M$ such that $m_i(x_i)=o$.
Then the images $m_i(R)$ are infinitely many rays that meet~$o$ and such that they do not only all contain rays starting at~$o$ but also contain geodesics to~$o$ of increasing lengths.
Thus, we obtain a geodesic double ray that meets~$o$.
Its subrays lie in~$\eta$.
Let $\mu\in\rand D$ contain its anti-subrays.
Since $d_h$ is visual, we have $d_h(\mu,\eta)>0$.

Let $Q$ be the anti-subray of that geodesic double ray that ends at~$o$.
Using geodesics $P_i$ from~$o$ to every vertex on~$Q$, we obtain a geodesic ray~$P$ that starts at~$o$ and lies in some hyperbolic boundary point.
Since there is a unique one that contains rays, we have $P\in\eta$.
We consider geodesic triangles with $o$, $o$ and $x$ on~$Q$ as end vertices and with $P_i$ and $xQo$ as non-trivial sides.
Since these triangles are $\delta$-thin, we obtain that all but at most the first $\delta$ vertices of~$P$ lie in $\cB^-_\delta(Q)$.
Similarly, we consider geodesic triangles with $o$, $x$ and $x$ on~$Q$ as end vertices, again with $P_i$ and $xQo$ as non-trivial sides.
These $\delta$-thin triangles show that all vertices of~$P$ lie in $\cB^+_\delta(Q)$.
Thus, we have $P\leq Q\leq P$ and hence $d_h(\eta,\mu)=0$ and $d_h(\mu,\eta)=0$, which implies $\eta=\mu$.
This is a contradiction since there cannot be a geodesic double ray from~$\eta$ to itself.
\end{proof}

Let us now show that these elements of infinite order define \qi\ embeddings.

\begin{thm}\label{thm_infOrderQI}
Let $M$ be a right cancellative finitely generated hyperbolic monoid.
If $g\in M$ has infinite order, then the \se\ induced by~$g$ on the Cayley digraph of~$M$ is not elliptic.

In particular, the embedding of $\langle g\rangle$ into $M$ is a \qiy.
\end{thm}

\begin{proof}
If $g$ induces an elliptic \se, then there exists $h\in M$ with $g^nh=h$ for some $n\in\N$.
Since $M$ is right cancellative, we obtain $g^n=1$ and thus $g$ has finite order.

The additional statement now follows directly from Proposition~\ref{prop_nonEllipticQIOfN}.
\end{proof}

It is easy to see that $\N\times\N$ is not a hyperbolic monoid.
However, this does not directly imply that it cannot be a submonoid of a hyperbolic monoid.
But for right cancellative finitely generated hyperbolic monoids, we will exclude this possibility.

\begin{thm}\label{thm_NoNtimesN}
Let $M$ be a right cancellative finitely generated hyperbolic monoid.
Then $M$ does not contain $\N\times\N$ as a submonoid.
\end{thm}

\begin{proof}
Let $D$ be a locally finite hyperbolic Cayley digraph of~$M$ and let $o$ be the vertex corresponding to the identity of~$M$.
Let $\varphi\colon\R\to\R$ be the function such that $D$ satisfies (B1) and (B2) for~$\varphi$.
Let $g\in M$ be of infinite order and let $R_g$ be the \qg\ in~$D$ starting at~$o$ that is a concatenation of geodesics joining $g^n$ to $g^{n+1}$ for all $n\in\N$.
First, we will show the following.
\begin{txteq}\label{itm_NoNtimesN_1}
There exists $\lambda\geq 0$ such that, for all $h\in C_M(g)$, the ray $h(R_g)$ eventually lies in the $\lambda$-out-ball of~$R_g$.
\end{txteq}
Since $h$ acts as a \se\ on~$D$ and all vertices have the same out-degree, the image $h(R_g)$ is a \qg, too.
Let $\gamma\geq 1$ and $c\geq 0$ such that $R_g$ and $h(R_g)$ are $(\gamma,c)$-\qg s.
By Lemma~\ref{lem_HyperDi_7.3}, there exists $\kappa$ such that all $(\gamma,c)$-\qg s lie within the $\kappa$-out- and $\kappa$-in-balls around geodesics with the same end vertices and vice versa.

Let $n,m\in\N$ with $0<n<m$.
Let $S_1$ be an $o$-$h$ geodesic and $S_2$ a $g^m$-$hg^m$ geodesic.
Let $P_1$ be an $o$-$g^m$ geodesic and let $P_2$ be an $h$-$hg^m$ geodesic.
Let $Q$ be an $o$-$hg^m$ geodesic.
Then there exists $x$ on~$P_2$ with $d(x,hg^n)\leq\kappa$.
The vertex on $Q$ of distance $d(o,h)+\delta+1$ from~$o$ must have a vertex of distance~$\delta$ from it on~$P_2$.
Since $h(R_g)$ is a \qg, we may assume that $n$ was chosen large enough such that this vertex lies on $hP_2x$.
Thus, there exists $u$ on~$Q$ with $d(u,v)\leq\delta$ for some vertex $v$ on $hP_2x$ such that the out-neighbour $u^+$ of~$u$ on~$Q$ satisfies $d(u^+,v^+)\leq\delta$ for some vertex $v^+$ on $xP_2g^mh$.
Applying Proposition~\ref{prop_hyperDi3.3}, we obtain
\[
d(v,v^+)\leq (2\delta+1)\varphi(\delta+1)
\]
and thus
\[
d(u,x)\leq \delta + (2\delta+1)\varphi(\delta+1).
\]

We may assume that $m$ was chosen large enough such that $d(u,hg^m)>\delta+d(g^m,hg^m)$.
Then there exists $w$ on~$P_1$ with $d(w,u)\leq\delta$.
By the choice of~$\kappa$, there is a vertex $y$ on~$oR_gg^m$ with $d(y,w)\leq\kappa$ and hence
\[
d(y,hg^n)\leq 2\kappa + 2\delta + (2\delta+1)\varphi(\delta+1).
\]
This proves~(\ref{itm_NoNtimesN_1}).

Let us suppose that there is $h\in C_M(g)$ of infinite order such that $\langle g,h\rangle$ is isomorphic to $\N\times\N$ in a canonical way.
Then (\ref{itm_NoNtimesN_1}) implies the existence of ${m_1,m_2,n_1,n_2\in\N}$ with $m_1<m_2$ such that
\[
h^{m_1}g^{n_1}=h^{m_2}g^{n_2}.
\]
This, however, is not possible in $\N\times\N$ unless $m_1=m_2$ and $n_1=n_2$.
This contradiction shows the assertion.
\end{proof}

We call a finitely generated hyperbolic monoid \emph{non-elementary} if it contains infinitely many hyperbolic boundary points.
Note that by~\cite[Corollary 11.2]{H-HyperDiBound} this is, for finitely generated cancellative hyperbolic monoids, equivalent to having more than two hyperbolic boundary points under the additional assumption that the hyperbolic boundary is a $T_1$-space.

\begin{thm}\label{thm_FreeRk2}
Let $M$ be a right cancellative non-elementary finitely generated hyperbolic monoid.
Then $M$ contains a free submonoid of rank~$2$.
\end{thm}

\begin{proof}
Let $D$ be a locally finite hyperbolic Cayley digraph of~$M$ and let $o$ be the vertex corresponding to the identity of~$M$.
By Propositions~\ref{prop_forCorOf14.1} and~\ref{prop_bpOfDist0}, there are infinitely many elements in $\cL(M)$.
Thus, Theorem~\ref{thm_freeSubmonoid} implies the assertion.
\end{proof}

\section{Growth of hyperbolic monoids}\label{sec_growth}

A \emph{growth function} is a monotone non-decreasing function $\N\to\N$, that is a function $f\colon \N\to\N$ with $f(t_1)\leq f(t_2)$ for all $t_1\leq t_2$.
For two growth functions $f,g$, we write $f\preccurlyeq g$ if there are $k,\ell\geq 1$ with $f(t)\leq kg(\ell t)$ for all $t\in\N$ and we write $f\sim g$ if $f\preccurlyeq g$ and $g\preccurlyeq f$.
This is an equivalence relation and the equivalence class of~$f$ is its \emph{growth type}.

Let $S$ be a semigroup with finite generating set $A$.
For $s\in S$, we denote by $\ell_A(s)$ the minimum length of a word in~$A$ the represents~$s$.
The function
\[
g_{A,S}\colon \N\to\N,\quad g_{A,S}(t)=|\{s\in S\mid \ell_A(s)\leq t\}|
\]
is the growth function of~$S$ with respect to~$A$.
The growth type of $g_{A,S}$ is independent of the particular finite generating set, so it is a semigroup invariant.
More generally, Gray and Kambites~\cite{GK-SemimetricSpaces} proved that the growth type is even a \qiy\ invariant for finitely generated semigroups.

We say that $S$ has \emph{exponential} growth type, if there exists $\theta>1$ such that 
\[
g_{A,S}(t)\geq \theta^t
\]
for all sufficiently large~$t$.
Note that free semigroups of rank at least~$2$ have exponential growth type.

We need the following lemma for the growth type of semigroups, which follows directly from the definition.

\begin{lem}\label{lem_growthSubsemigroups}
Let $S$ be a semigroup and $S'$ a subsemigroup of~$S$.
If $f$ and~$g$ are growth functions of~$S'$ and~$S$, respectively, then $f\preccurlyeq g$.\qed
\end{lem}

\begin{thm}\label{thm_Growth}
Let $M$ be a right cancellative non-elementary finitely generated hyperbolic monoid.
Then $M$ has exponential growth.
\end{thm}

\begin{proof}
By Theorem~\ref{thm_FreeRk2}, $M$ contains a free subsemigroup of rank~$2$.
This has exponential growth type and thus, the growth type of~$M$ is at least exponential by Lemma~\ref{lem_growthSubsemigroups} and thus it is exponential.
\end{proof}

\section{Self-embeddings of graphs}\label{sec_graphs}

In~\cite{H-SelfEmbeddingsTrees}, \se s of trees were analysed and in \cite[Section 5.1]{H-SelfEmbeddingsTrees} the general problem to extend the results to graphs was posed.
While for \se s of trees a result says that each of them fixes either a finite vertex set setwise or at most two ends pointwise, see \cite[Corollary 2.3]{H-SelfEmbeddingsTrees}, we will prove in this section that the corresponding result for graphs is wrong in general.
For that, we will construct in Example~\ref{ex_seMayFixArbitrarilyManyBP_graph}, for every $n\in\N$, a graph with precisely $n$ ends and a \se\ of that graph that fixes all ends pointwise but no finite vertex set setwise.
This example is just the underlying undirected graph of Example~\ref{ex_seMayFixArbitrarilyManyBP}.
But before we construct that, let us fix some notions for graphs.

A \emph{ray} is a one-way infinite path and a \emph{double ray} is a two-way infinite path.
Two rays in a graph~$G$ are \emph{equivalent} if there exists for every finite vertex set $S\sub V(G)$ a unique component that contains all but finitely many vertices from both rays.
This is an equivalence relation whose classes are the \emph{ends} of~$G$.

\begin{ex}\label{ex_seMayFixArbitrarilyManyBP_graph}
Let $n\in\N$.
Let $x_0x_1\ldots$ be a ray and for every $1\leq i\leq n-1$ let $\ldots y_{-1}^iy_0^iy_1^i\ldots$ be a double ray.
Let $G$ be the graph obtained from the disjoint union of this ray and these double rays with $x_jy_j^i$ as additional edges for every $j\in\N$ and every $0\leq i\leq n-1$.
Then $G$ has precisely $n$ ends.
Let $g$ be the \se\ that maps every $x_j$ to $x_{j+1}$ and every $y_j^i$ to $y_{j+1}^i$.
This \se\ fixes each end but no finite vertex set setwise.
\end{ex}

\providecommand{\bysame}{\leavevmode\hbox to3em{\hrulefill}\thinspace}
\providecommand{\MR}{\relax\ifhmode\unskip\space\fi MR }
\providecommand{\MRhref}[2]{%
  \href{http://www.ams.org/mathscinet-getitem?mr=#1}{#2}
}
\providecommand{\href}[2]{#2}

\end{document}